\title{A non-trivial upper bound on the threshold bias of the Oriented-cycle game}
\author{
\quad{Dennis Clemens
\thanks{Institut f\"ur Mathematik, Technische Universit\"at Hamburg--Harburg, Germany.
Email: dennis.clemens@tuhh.de. Work done at FU Berlin while supported by DFG, project SZ 261/1-1.}}
\quad{Anita Liebenau
\thanks{School of Mathematical Sciences, Monash University, Clayton 3800, Australia. Email: anita.liebenau@monash.edu.
Work done while at FU Berlin, supported by the Berlin Mathematical School.}}}

\documentclass[11pt,a4paper]{article}
\usepackage{amsmath, amsthm,amssymb,latexsym,color,epsfig,a4, enumerate}
\usepackage{bbm, tikz}
\usepackage{caption, subcaption}
\usepackage{a4wide}

\newtheorem{theorem}{Theorem}[section]
\newtheorem{lemma}[theorem]{Lemma}
\newtheorem{proposition}[theorem]{Proposition}
\newtheorem{observation}[theorem]{Observation}
\newtheorem{corollary}[theorem]{Corollary}
\newtheorem{conjecture}[theorem]{Conjecture}
\newtheorem{problem}[theorem]{Problem}
\newtheorem{definition}[theorem]{Definition}

\def\cA{{\cal A}}

\def\cC{{\cal C}}
\def\cF{{\cal F}}
\def\cG{{\cal G}}
\def\cH{{\cal H}}
\newcommand{\cL}{\ensuremath{\mathcal{L}}}

\def\cP{{\cal P}}

\renewcommand{\leq}{\leqslant}
\renewcommand{\geq}{\geqslant}
\newcommand{\se}{\subseteq}

\newcommand{\eps}{\ensuremath{\varepsilon}}

\newcommand{\rest}{\ensuremath{V\setminus (A\cup B)}}
\newcommand{\restStrich}{\ensuremath{V\setminus (A'\cup B')}}

\newcommand{\sm}{\ensuremath{\setminus}}
\newcommand{\size}[1]{\left| #1 \right|}

\newcommand\mytabs{\hspace*{10pt}\=\hspace{1.5cm}\=\hspace{4.8cm}}
\newenvironment{mytabbing}[1][\mytabs]
  {\begin{tabbing}#1\kill}
  {\end{tabbing}}

\newcommand*\widebar[1]{%
  \hbox{%
  \vbox{%
    \hrule height 0.5pt 
    \kern0.5ex
    \hbox{%
      \kern-0.1em
      \ensuremath{#1}%
      \kern-0.1em
      }%
    }%
}%
}

\newcommand{\tikzoverset}[2]{%
  \tikz[baseline=(X.base),inner sep=0pt,outer sep=0pt]{%
    \node[inner sep=0pt,outer sep=0pt] (X) {$#2$}; 
    \node[yshift=1pt] at (X.north) {$#1$};
}}

\newcommand\back[1]{
  \tikzoverset{\leftarrow}{#1}
}

\newcommand{\Dback}{\back{D}}
\newcommand{\eback}{\back{e}}

\definecolor{darkgreen}{RGB}{0,100,0}
\definecolor{nicegreen}{RGB}{0,204,0}
\definecolor{myred}{RGB}{220,24,10}
\definecolor{myorange}{RGB}{255,165,0}

\begin{document}
\maketitle

\begin{abstract}
In the Oriented-cycle game, introduced by Bollob\'as and Szab\'o \cite{bs1998}, 
two players, called OMaker and OBreaker, alternately direct edges of $K_n$. 
OMaker directs exactly one previously undirected edge, whereas OBreaker is allowed to direct 
between one and $b$ previously undirected edges. 
OMaker wins if the final tournament contains a 
directed cycle, otherwise OBreaker wins.
Bollob\'as and Szab\'o \cite{bs1998} conjectured that for a bias as large as $n-3$ OMaker 
has a winning strategy if OBreaker must take exactly $b$ edges in each round. 
It was shown recently by Ben-Eliezer, Krivelevich and Sudakov \cite{bks2012}, that OMaker 
has a winning strategy for this game whenever $b< n/2-1$. 
In this paper, we show that OBreaker has a winning strategy whenever 
$b> 5n/6+1$. Moreover, in case OBreaker is required to direct 
exactly $b$ edges in each move, we show that OBreaker wins for $b\geq 19n/20$, 
provided that $n$ is large enough. This refutes the conjecture by Bollob\'as and Szab\'o. 
\end{abstract}


\section{Introduction} 
We consider biased orientation games, 
as discussed by Ben-Eliezer, Krivelevich and Sudakov in \cite{bks2012}. 
In orientation games, the board consists of the edges of the complete graph 
$K_n$. In the $(a:b)$ orientation game, the two players called OMaker and 
OBreaker, alternately direct previously undirected edges. 
OMaker starts, and in each round, OMaker directs between one and $a$ edges, 
and then OBreaker directs between one and $b$ edges. 
At the end of the game, the final graph is a tournament on $n$ vertices. 
OMaker wins the game if this tournament has some predefined property $\cP$. 
Otherwise, OBreaker wins. 

Orientation games can be seen as a relative of $(a:b)$ Maker--Breaker
games, played on the complete graph $K_n$. The game is played by two players,
Maker and Breaker, who alternately claim $a$ and $b$ edges, respectively. 
Maker wins if the subgraph consisting of her edges satisfies some given monotone-increasing property $\cP$.
Otherwise, Breaker wins. Maker--Breaker games have been widely studied  
(cf.~\cite{Beck82}, \cite{Beck85}, \cite{Beck94}, \cite{bp2005}, \cite{ec1978}, \cite{ES1973}, \cite{gs2009}, \cite{k2011}), 
and it is quite natural to translate typical questions about Maker--Breaker games
to orientation games. 

For instance, Beck \cite{BeckBook} studied the so-called {\em Clique game}, 
proving that in the $(1:1)$ Maker--Breaker game, the largest clique that
Maker is able to build is of size $(2-o(1))\log_2(n)$.
Motivated by this result, an orientation game version of the Clique game
is considered in \cite{cgl2013}: Given a tournament $T_k$
on $k$ vertices, it is proven there that for $k\leq (2-o(1))\log_2(n)$  
OMaker can ensure that $T_k$ appears in the final tournament, 
while for $k\geq (4+o(1))\log_2(n)$ OBreaker always can prevent
a copy of $T_k.$

In this work we only consider orientation games with $a=1$. 
We refer to 
the $(1:1)$ orientation game as the {\em unbiased} orientation game, 
and the $(1:b)$ orientation game as the {\em $b$-biased orientation game} when $b>1$. 
Increasing $b$ can only help OBreaker (since OBreaker can choose to direct fewer than $b$ edges per round)  
so the game is {\em bias monotone}. 
Therefore, 
any such game (besides degenerate games where $\cP$ is a property
that is satisfied by every or by no tournament on $n$ vertices) has a {\em threshold} $t(n,\cP)$ such that OMaker wins 
the $b$-biased game when $b\leq t(n,\cP)$ and OBreaker wins the game when 
$b> t(n,\cP)$.

In a variant, OBreaker is required to direct exactly $b$ edges in each round. 
We refer to this variant as the {\em strict} $b$-biased orientation game, 
where {\em the strict rules apply}. 
Accordingly, we say the {\em monotone rules apply} in the game we defined above --  when OBreaker is 
free to direct between one and $b$ edges. 
Playing the exact bias in every round may be disadvantageous for OBreaker, 
so the existence of a threshold as for the monotone rules is unclear in general. 
We therefore define 
$t^+(n,\cP)$ to be the largest value $b$ such that OMaker has a strategy to win 
the strict $b$-biased orientation game, 
and $t^-(n,\cP)$ to be the largest integer such that for {\em every} $b\leq t^-(n,\cP)$, 
OMaker has a strategy to win the strict $b$-biased orientation game.
(The definition of these two threshold functions
is motivated by the study of Avoider-Enforcer games, cf.~\cite{hks2007}, \cite{hkss2010}.)
Trivially, $t(n,\cP)\leq  t^-(n,\cP) \leq t^+(n,\cP)$ holds.

The threshold bias $t(n,\cP)$ was investigated by Ben-Eliezer, Krivelevich and Sudakov \cite{bks2012} for 
several orientation games. They showed, for example, that $t(n,\cH)= (1+o(1)) n/\ln n$, 
where $\cH$ is the property to contain a directed Hamilton cycle.
In general, very little is known about the relation between all three parameters in question. 
It is not even clear whether $t^-(n,\cH)$ and $t^+(n,\cH)$ need to be distinct values.

In this work, we study the {\em Oriented-cycle game}, introduced by Bollob\'as and Szab\'o \cite{bs1998},  
which is an orientation game where $\cP=\cC$ is the property of containing a directed cycle. That is, OMaker wins
if the final tournament contains a directed cycle, and OBreaker wins if the final tournament 
is transitive. 
The Maker-Breaker variant of this game is studied in \cite{bp2005}, where it 
is shown that Maker has a strategy to claim a cycle in the $(1:b)$ game if and only if $b<\left\lceil n/2 \right\rceil - 1$.

The strict version of the Oriented-cycle game was studied by Alon (unpublished, cf.~\cite{bs1998}), 
and later by Bollob\'as and Szab\'o \cite{bs1998}. 
For an upper bound on the threshold bias in the orientation game, 
it is observed in \cite{bs1998} that $t^+(n,\cC)\leq n-3$. 
Indeed, with a short case distinction, it can be verified that for $b\geq n-2$, 
OBreaker can always ensure that immediately after each round there exists
a subset $\{v_1,\ldots,v_k,v_{k+1}\}\subseteq V=V(K_n)$ such that for every $1\leq i\leq k$ and $v\in V\setminus \{v_1,\ldots,v_i\}$ the edge $v_iv$ is directed 
from $v_i$ to $v$; and every directed edge in $V\setminus \{v_1,\ldots,v_k\}$ starts in $v_{k+1}$.
If these properties hold, there is neither a directed cycle nor an edge that could close such a cycle, i.e.~OBreaker wins.
We refer to this strategy as the {\em trivial strategy}.

For a lower bound, it is proved in \cite{bs1998} that $t^+(n,\cC) \geq \lfloor (2-\sqrt{3}) n\rfloor $. 
Moreover, they remark that the proof also works for the monotone rules, 
implying that $t(n,\cC) \geq \lfloor (2-\sqrt{3}) n\rfloor$. 
Finally,  they conjectured that $t^+(n,C)=n-3$. 
In \cite{bks2012}, Ben-Eliezer, Krivelevich and Sudakov improve the lower bound 
and show that for $b\leq n/2-2$, OMaker has a strategy guaranteeing a cycle in the 
(monotone) $b$-biased orientation game, i.e. $t(n,\cC)\geq n/2-2$. 
In the main result of this paper we refute the conjecture of Bollob\'as and Szab\'o and give a 
strategy for OBreaker to prevent cycles when $b\geq 19n/20$ and $n$ large enough. 
 
\begin{theorem} \label{OrientedCycleStrict}
For large enough $n$, $t^+(n,C)\leq 19n/20-1$. 
\end{theorem}
In the monotone game our strategy simplifies and gives a winning strategy already when $b=5n/6+2$.

\begin{theorem} \label{OrientedCycleMain}
$t(n,C)\leq 5n/6+1.$
\end{theorem}
The remaining part of the paper is organized as follows. First we introduce some
general notation and terminology. In Section \ref{monotone}, we introduce
some necessary concepts and prove Theorem \ref{OrientedCycleMain}.
In Section \ref{strict}, we describe a strategy for OBreaker in the strict game, 
we prove that this strategy constitutes a winning strategy in Sections \ref{sec:StrictStageI} 
and \ref{sec:StrictStageII}. 
We finish the paper with some concluding remarks.

\subsection{General notation and terminology}
Let $V = [n]$ and let $D \subseteq V\times V$ be a digraph. 
We call elements $(v,w)\in D$ {\em arcs} and the underlying 
set $\{v,w\}$ {\em a pair} or {\em an edge}. 
An arc $(v,v)$ is called a {\em loop} and $(v,w)$ is called 
the {\em reverse arc} for $(w,v)$. 
In this work we are only concerned with {\em simple} digraphs, 
without loops and reverse arcs. 
For an arc $e\in D$, we write $e^+$ for its tail and $e^-$ for 
its head, i.e.~$e=(e^+,e^-)$. 
By $\eback$ we denote the reverse arc of $e$.
For a subdigraph $S\subseteq D$, we denote by $S^+$ the set 
of all tails $e^+$ for $e\in S$, and by $S^-$ the set of all heads 
$e^-$ for $e\in S$. 
It is convenient to denote by $\Dback$ the set of all 
reverse arcs of $D$, which we call the {\em dual of $D$}, 
that is $\Dback := \{\eback : e\in D\}$. 
Moreover, the set $\cA (D) := (V\times V)\setminus (D\cup \Dback \cup \cL)$ 
denotes the set of all {\em available arcs}, where 
$\cL =\{(v,v): v\in V\}$ is the set of all loops. 
Note that $\cA(D)$ is symmetric, i.e. if $(v,w)\in \cA(D)$ then also $(w,v)\in \cA(D)$. 
We generalize the notation of an arc and say the 
$k$-tuple $(v_1,\ldots,v_k)$ {\em induces a transitive tournament in $D$}, 
if for all $1\leq i<j \leq k$ we have that 
$(v_i,v_j)\in D$. 
We call a vertex $v\in V$ a {\em source of $D$} if for all $u\in V$, 
$(u,v)\not\in D$. Clearly, if $D$ is a transitive tournament, it has a unique vertex which is a source.  
For two disjoint sets $A,B\subseteq V$ we call the pair $(A,B)$ a {\em uniformly directed biclique}, for short 
UDB, if for all $v\in A$, $w\in B$ we have that $(v,w)\in D$.
We say the sequence $P=(e_1,\ldots,e_k)$ is a 
{\em directed path} (or simply a {\em path}) in $D$ 
if all $e_i\in D$ and for all $1\leq i <k$ we have that $e_i^- = e_{i+1}^+$. 
In this case we say that $P$ is a $e_1^+$-$e_k^-$-path.
We also write $P=v_1,\ldots,v_{\ell}$ to denote the
path $P=(e_1,\ldots,e_{\ell-1})$ where $e_i=(v_i,v_{i+1})$. 

In our proofs we are concerned how $D$ behaves on certain 
subsets of the vertices. For a subset $A\subseteq V$ we denote 
by $D(A)$ the directed subgraph of $D$ of arcs spanned by $A$. 
Formally, $D(A):=D\cap (A\times A)$. For two (not necessarily disjoint)
sets $A,B\subseteq V$, we set $D(A,B):=D\cap (A\times B)$
to be the set of those edges in $D$ that start in $A$ and end in $B$.
To shorten the notation, we also set $D(v,B):=D(\{v\},B)$ and
$D(A,v):=D(A,\{v\})$ for $v\in V$. Moreover, to describe the sizes of these edge sets,
we let $e_D(A):=\size{D(A)},\ e_D(A,B):=\size{D(A,B)},\ 
e_D(v,B):=\size{D(v,B)}$ and $e_D(A,v):=\size{D(A,v)}$.
Given the digraph $D$, we also might want to delete or
add some edge $e$. To simplify the notation, we write
$D+e:=D\cup \{e\}$ and
$D-e:=D\setminus \{e\}$.

Recall that the Oriented-cycle game is played on the edge set of 
$K_n$ where we assume that $V = V(K_n) =[n]$.
We say a player {\em directs the edge $(v,w)$} 
if s/he 
directs the pair $\{v,w\}$ from $v$ to $w$. 
That is, the player chooses the arc $(v,w)$ to belong to the final digraph, 
and dismisses the arc $(w,v)$ from the board.   
After some round $r$, we shall refer to $D\subseteq V\times V$ as the 
sub-digraph of already directed edges (arcs) by either player. 
We say a player {\em closes a cycle in $D$ by directing some 
edge $(v,w)$} if there exists a $w$-$v$-path in $D$. 
Note that if a player can close a cycle in $D$, then 
s/he 
can close a triangle (consider the shortest cycle a player can close, 
and consider any chord).

\section{The Oriented-cycle game -- monotone rules} \label{monotone}

There are two essential concepts to our proof, the aforementioned UDBs and 
{\em $\alpha$-structures} which we define below. 
A UDB is a complete bipartite digraph where all the edges are directed 
in the same direction (i.e.~from $A$ to $B$). 
OBreaker's goal is to create a UDB $(A,B)$ such that both parts fulfil 
$|A|,|B| \leq b$ and $A\cup B = V$. 
Suppose the following situation occurs at some point. 
There is a partition $A\cup B = V$ such that the pair 
$(A,B)$ forms a UDB in $D$, both parts fulfil 
$|A|,|B| \leq b$, and both sets $A$ and $B$ are empty 
(i.e.~$D(A)=D(B) =\emptyset$). OBreaker 
could then follow the trivial strategy inside $A$ and $B$, respectively 
(as OBreaker wins on $K_{b+2}$), 
even when the strict rules apply. 

However, 
if OBreaker's strategy is to build such a UDB,  
OMaker can direct edges inside 
these sets, and OBreaker needs to control those. 
Moreover, to optimise the bias, OBreaker should be able to 
control those edges inside $A$ and $B$ with as few edges as possible. 
To handle this obstacle, we introduce certain structures which we 
call {\em $\alpha$-structures} and a way  
to incorporate new (i.e.~OMaker's) edges into an existing 
$\alpha$-structure. 
Before we move on to study these special 
structures  let us mention that the idea of 
building a big UDB quickly comes up again in 
the proof of Theorem \ref{OrientedCycleStrict} in the subsequent sections. 
However, the requirement of directing exactly $b$ edges in each move 
puts some serious restrictions on the power of $\alpha$-structures. 
So for the strict rules, we then consider only special $\alpha$-structures, 
that are more robust to adding more edges.

\subsection{$\alpha$-structures}

The definition of an $\alpha$-structure looks quite technical 
at first sight. So let us motivate the idea behind it. 

Suppose OMaker's strategy is to build a long path first. 
(This indeed is the strategy for OMaker in the so far best-known 
lower-bound proof in \cite{bks2012}.)
Let $P = (e_1,\ldots,e_k)$ be a directed path 
of length $k$ in $D$ with arcs $e_i=(v_i,v_{i+1})$, 
and suppose OMaker enlengthens $P$ by directing 
an edge $(v_{k+1},w)$ for some $w\in V$. 
Then all the pairs $\{w,v_i\}$ for $1\leq i \leq k$ 
constitute potential threats as directing any $(w,v_i)$ 
would close a cycle 
(we call such pairs {\em immediate threats}). 
So OBreaker better direct all edges $(v_i,w)$ in his next 
move; we call this {\em closing} immediate threats. 
This way, OBreaker fills up the missing arcs of an evolving 
transitive tournament with {\em spine} $e_1,\ldots,e_k$.
Then formally, OBreaker sets $v_{k+2}:= w$, $e_{k+1}:= (v_{k+1},w)$ 
and directs all edges $(e_i^+,w)$ for $i\leq k$.
Clearly, as long as there are isolated vertices, OMaker could 
follow this strategy and increase the number of threats that 
OBreaker {\em has to close immediately} by one in each move. 

A na\"ive strategy for OBreaker would be to close immediate threats 
only. However, then OMaker could claim two vertex-disjoint  
paths of linear length, say $P_1=v_1,\ldots,v_{\eps n}$ and 
$P_2=w_1,\ldots,w_{\eps n}$, for some $\eps >0$,
and OBreaker would not claim edges between these two paths.
Directing the 
arc $(v_{\eps n},w_1)$ then, OMaker could suddenly create $(\eps n)^2>b$  immediate  
threats, which OBreaker cannot close in the next move. 

By defining the $\alpha$-structure we prevent such a situation. 
Moreover, 
we show that ``building a long path" is the 
best possible strategy for OMaker in the following sense: 
No matter how OMaker plays, 
OBreaker has a strategy such that in round $r$, 
he has to direct at most $r$ edges to close immediate threats. 

\begin{definition}\label{def:alpha}
Let $V$ be a set of vertices, let $r$ be a nonnegative integer, and let $D\subset (V\times V) \setminus \cL$ 
be a digraph without loops and reverse arcs. 
Then $D$ is called an {\em $\alpha$-structure of rank $r$} 
if there exist $k\leq r$ arcs $e_1,\ldots,e_k\in D$ such that 
for all $u,v\in V$:  
$$(u,v)\in D \text{ if and only if } (u,v) = (e_i^+,e_j^-) \text{ for some } 1\leq i\leq j \leq k.$$
The arcs $e_1,\ldots,e_k$ are called {\em decisive arcs of the 
$\alpha$-structure $D$.} 
\end{definition}

 \begin{figure}[bp]
 \centering
\phantom{asdfasdfasdfaasdfasfasdfad}
\begin{subfigure}{0.3\textwidth}
	\centering
	 \includegraphics
	 {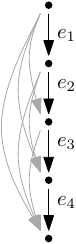}
\end{subfigure} 
\quad
 \begin{subfigure}{0.3\textwidth}
 	\centering
	\includegraphics
	{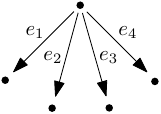}
 \end{subfigure}
\qquad
 \begin{subfigure}{0.3\textwidth}
 	\centering
	\includegraphics
	{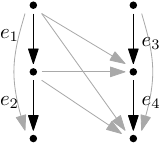}
 \end{subfigure}
 \caption{Three $\alpha$-structures of rank four that are not $\alpha$-structures of rank three.}
 \label{fig1}
 \end{figure}

In Figure \ref{fig1}, we give three simple examples of $\alpha$-structures. 
In our strategy that we describe later, the arcs $e_1,\ldots,e_k$ are edges directed by 
OMaker (though not necessarily in that order).
Note that the arcs $e_1,\ldots,e_k$ uniquely determine the $\alpha$-structure $D$. 
In particular, $D^+=\{e_1^+,\ldots,e_k^+\}$ and $D^-=\{e_1^-,\ldots,e_k^-\}$.

The condition $k\leq r$ in the definition above
might seem somewhat artificial at first sight. 
All the properties about $\alpha$-structures in this subsection 
are still true (or slight variations of them) if we require $k=r$ in the definition. 
However,  in the next subsection, this 
relaxed definition makes it easier to handle OBreaker's strategy.

Now, let us capture some immediate facts about $\alpha$-structures. 
The first proposition states that an $\alpha$-structure is {\em self-dual} in the following sense. 
\begin{proposition}\label{dualMon}
$D$ is an $\alpha$-structure of rank $r$ if and only if 
$\Dback$ is an $\alpha$-structure of rank $r$. 
\end{proposition}
\begin{proof}
If $e_1,\ldots,e_k$ are decisive arcs in  $D$, then $\back{e_k},\ldots,\back{e_1}$ are 
decisive arcs in $\Dback$.
\end{proof}

The next lemma says that restricting to induced subgraphs does not destroy $\alpha$-structures. 
\begin{lemma}\label{obs:trivialFacts}
Let $D$ be an $\alpha$-structure of rank $r$ on vertex set $V$. 
Then for any subset $V'\subseteq V$, we have that $D(V')$ is an 
	$\alpha$-structure of rank 
$r$.

\end{lemma}
\begin{proof}
Let $v\in V.$ It is enough to prove the statement for $V'=V\setminus \{v\}.$
Let $S_0=\{e_1,\ldots, e_k\}$ be a set of decisive arcs for $D$, with $k\leq r$,
as given by Definition \ref{def:alpha}. In the following we construct a
set $S_k$ iteratively which then turns out to be a set of decisive arcs of $D':=D(V').$ 
For every $i=1,\ldots, k$ do the following: 
\begin{itemize}
\item[1.] If $v\notin \{e_i^+,e_i^-\}$, then set $f_i:=e_i$ and $S_i:=S_{i-1}.$
\item[2.] If $v=e_i^-$ and there is an arc $g\in D'\setminus S_{i-1}$ with $g^+=e_i^+$
 such that $g$ {\em is forced by $e_i$} (there exists $j>i$ such that $g^-=e_j^-$),
then set $f_i:=g$ and $S_i:=S_{i-1}-e_i+f_i.$
\item[3.] If $v=e_i^+$ and there is an arc $g\in D'\setminus S_{i-1}$ with $g^-=e_i^-$
such that $g$ {\em is forced by $e_i$} (there exists $s<i$ such that $g^+=e_s^+$), 
then set $f_i:=g$ and $S_i:=S_{i-1}-e_i+f_i.$
\item[4.] In any other case set $S_i:=S_{i-1}-e_i.$
\end{itemize}

Let $I\subseteq [k]$ be the index set of those arcs $e_i$ removed in Case 4.  
We claim that $D'$ is an $\alpha$-structure with decisive arcs $S_k=\{f_i : i\not\in I\}$.  
 Let $u,w \in V'$. We need to show that 
 $(u,w)\in D'$ if and only if $(u,w)=(f_i^+,f_j^-)$ for some $1\leq i\leq j \leq k$ with 
 $i,j\not\in I$.  

First, assume that $(u,w)\in D'$. 
If $(u,w)=f_i=(f_i^+,f_i^-)$ for some $i\notin I$ then we are done. 
So, assume that $(u,w)\notin S_k$. Note that this implies that $(u,w)\notin S_i$ for all $i\leq k.$
Since $D$ is an $\alpha$-structure, $(u,w)=(e_i^+,e_j^-)$ for some $1\leq i\leq j\leq k.$ 
Consider the algorithm at iteration $i$ when $e_i$ is handled. 
If $e_i^-\neq v$, then Case 1 applies for $e_i$. That is, $e_i= f_i\in S_k$ and thus $u=f_i^+$ for $i\notin I$. 
If $e_i^-= v$, then Case 2 applies for $e_i$ since $(u,w)\notin S_{i-1}$ is such an arc {\em forced by $e_i$} 
(though the algorithm may choose $g$ different from $(u,w)$). 
Then $u=e_i^+=g^+=f_i^+$ and $i\notin I$. 
Consider now the algorithm at iteration $j$ when $e_j$ is handled. 
Similarly, either Case 1 or Case 3 applies for $e_j$, depending whether $v\neq e_j^+$ or $v=e_j^+$.  
Thus, $e_j^-=f_j^-$ and $j\notin I$. 
Therefore, $(u,w)=(f_i^+,f_j^-)$ for some $1\leq i\leq j \leq k$ with 
 $i,j\not\in I$.  

Let now $(u,w)=(f_i^+,f_j^-)$ for some $1\leq i\leq j \leq k$ with 
 $i,j\not\in I$, and assume that $(u,w)\not\in D'$. 
Since $f_\ell \in D'$ for all $1\leq \ell\leq k$, $\ell\not\in I$, both $u$ and $w$ must be distinct from $v$. 
That is we can in fact assume that  $(u,w)\not\in D$. 
However, we claim that $f_i^+= e_s^+$ and $f_j^-=e_t^-$ for some $s\leq t$ which implies that 
$(u,w)=(e_s^+,e_t^-)\in D$, since $D$ is an $\alpha$-structure, a contradiction.  
Indeed, consider the algorithm during iteration $i$ when $f_i$ is determined. 
Note that Cases 1, 2 or 3 must apply since we assume that $i\not\in I$. 
In Case 1, $f_i^+ = e_i^+$. In Cases 2 and 3, $e_i$ contains $v$ and thus is replaced by $f_i\neq e_i$. 
In Case 2, $f_i^+ = g^+=e_i^+$ by definition. In Case 3, $f_i=g$ is forced by $e_i$, that is, $f_i^+ = e_s^+$ 
for some $s<i$. In all three cases, $f_i^+ = e_s^+$ for some $s\leq i$. 
Finally, consider the algorithm during iteration $j$ when $f_j$ is determined. 
Again, Case 4 cannot apply for $e_j$. If Cases 1 or 3 apply, then $f_j^-=e_j^-$. 
If Case 2 applies, then $f_j^-=e_t^-$ for some $t>j$. 
Therefore, $(u,w)=(f_i^+,f_j^-)=(e_s^+,e_t^-)$ where $s\leq i\leq j\leq t$ and the proof is finished. 
\end{proof}

Before we prove that $\alpha$-structures are indeed acyclic, let us look 
at paths in them. The next lemma roughly says that any path in an $\alpha$-structure 
can be controlled through its decisive arcs.  
An illustration can be found in Figure \ref{fig:PathProp}.

\begin{figure} [bp]
\centering
\includegraphics[width=0.7\textwidth]{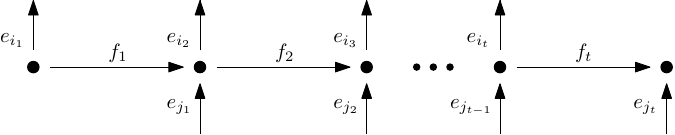}
 \caption{Illustrating Proposition \ref{crucialProp}.}
 \label{fig:PathProp}
 \end{figure}

\begin{proposition}\label{crucialProp}
Let $D$ be an $\alpha$-structure of rank $r$ with decisive arcs 
$e_1,\ldots,e_k$ and $k\leq r$. Let $P=(f_1,\ldots,f_t)$ be a path in $D$. Then 
there are decisive arcs $e_{i_1}, e_{j_1}, e_{i_2}, e_{j_2},\ldots , e_{i_t}, e_{j_t}$ 
(not necessarily distinct) 
such that $f_s = (e_{i_s}^+, e_{j_s}^-)$  for all $1\leq s\leq t$ and such that 
$i_1\leq j_1\leq i_2\leq j_2\leq \cdots\leq i_t\leq j_t$.
\end{proposition}
\begin{proof}
For every $1\leq s\leq t$, the arc $f_s$ is an element of $D$, 
hence the existence of decisive arcs $e_{i_s}, e_{j_s}$ with $i_s\leq j_s$ and  
$f_s = (e_{i_s}^+, e_{j_s}^-)$ follows. 
Now, suppose there was some $s\in [t-1]$ such that 
$j_s>i_{s+1}$. Then the arc $(f_{s+1}^+,f_s^-)=(e_{i_{s+1}}^+,e_{j_s}^-)$ 
would be a loop in $D$, a contradiction.   
\end{proof}

The following property is crucial for our orientation game. 
Recall that $\cA (D) = (V\times V)\setminus (D\cup \Dback \cup \cL)$ denotes the 
set of available arcs.
\begin{proposition}\label{prop:noCycles}
If $D$ is an $\alpha$-structure, then for every available $e\in \cA (D)$ 
we have that $D+e$ is acyclic. 
\end{proposition}
\begin{proof}
Suppose there was a path 
$P=(f_1,\ldots,f_t)$ in $D$ and an edge $e\in D\cup \cA(D)$ such that $P+e$ forms a directed cycle. 
Let $e_{i_1}, e_{j_1}, e_{i_2}, e_{j_2},\ldots , e_{i_t}, e_{j_t}$ be given by Proposition \ref{crucialProp}. 
Then since $i_1 \leq j_t$ we have that $(f_1^+,f_t^-)=(e_{i_1}^+,e_{j_t}^-) \in D$. 
So, $e=(f_t^-, f_1^+) \in \Dback$ and hence not available, a contradiction. 
\end{proof}
\noindent
In the light of our orientation game, we pin down the following 
important implication. 
\begin{corollary}\label{cor:noCycles}
For some subset $V' \subseteq V$, suppose that in the Oriented-cycle game, 
OBreaker maintains that $D(V')$ is an $\alpha$-structure (of some rank $r$). Then 
there is no cycle in $D(V')$ and 
OMaker cannot close a cycle inside $V'$ in her next move.
\end{corollary}

In order for OBreaker to maintain an $\alpha$-structure on some subset 
of the vertices we need to know how to incorporate OMaker's edge 
into such a structure. The following is one of the key lemmas 
that we use for OBreaker's strategy. 

\begin{lemma}\label{alphaProc}
Let $D$ be an $\alpha$-structure of rank $r$ on vertex set $V$, 
and let $e\in \cA(D)$ be an available arc.
Then there exists  
a set $\{f_1,\ldots,f_t\}\subseteq \cA(D)$ of 
at most $\min \{r, \size{V}-2\}$ available arcs 
 such that 
$D'= D\cup\{e, f_1,\ldots,f_t\}$ is an $\alpha$-structure of rank $r+1$.
Moreover, $D'^+=D^+\cup\{e^+\}$ and $D'^-=D^-\cup\{e^-\}$. 
\end{lemma}

Before we prove the lemma, we need one more definition. 
Let $D$ be an $\alpha$-structure 
with decisive arcs $S=\{e_1,\ldots,e_k\}$, and let $x\in V$. 
We set
\begin{align*}
In (x) &:= 
	\Big\{ e_i\, :\, \text{there exists a path } P = (e_i,e_{j_1},\ldots,e_{j_m})
	\text{ s.t. } x = e_{j_m}^-\Big\}, \\
Out (x) &:= 
	\Big\{ e_i\, :\, \text{there exists a path } P = (e_{j_1},\ldots,e_{j_m},e_i)
	\text{ s.t. } x = e_{j_1}^+\Big\}. 
\end{align*}
The following proposition is rather simple.
\begin{proposition}\label{AboveBelow}
Let $D$ be an $\alpha$-structure 
with decisive arcs $e_1,\ldots,e_k$. Further, 
let $x,y \in V$ be vertices such that 
$(x,y)\in \cA(D)$. Then 
for all $e_i \in In(x)$, $e_j \in Out(y)$: $i<j$. 
In particular, 
$In(x)\cap Out(y) = \emptyset$.
\end{proposition}

\begin{proof}
Let $e_i \in In(x)$, $e_j \in Out(y)$ 
and let $P_i$ be the corresponding $e_i^+$-$x$-path starting with $e_i$, and 
let $P_j$ be the corresponding $y$-$e_j^-$-path ending with $e_j$. 
Assume $j\leq i$. By definition of an $\alpha$-structure, this implies that 
the arc $(e_j^+,e_i^-) \in D$. 
But then the concatenation of $(P_j-e_j)$, $(e_j^+,e_i^-)$ and $(P_i-e_i)$ 
is a directed walk in $D$ from $y$ to $x$, i.e.~contains a directed path from
$y$ to $x$. 
By Proposition \ref{prop:noCycles}, $(x,y)$ is not available, a contradiction. 
\end{proof}
\noindent
We are now ready to prove the above lemma.  
\begin{proof}[Proof of Lemma \ref{alphaProc}]
Let $D$ be an $\alpha$-structure of rank $r$ on a vertex set $V$, 
let $S=\{e_1,\ldots,e_k\}$ be a set of decisive arcs with $k\leq r$, 
and let $e=(v,w)\in \cA(D)$ be an available arc.  
Set 
$\ell := \min \{i: e_i \in Out(w) \}$ if $Out(w) \neq \emptyset$, 
and $\ell := k+1$ otherwise. 
For all $i<\ell$, set $f_i:=(e_i^+,w)$, and for all $i\geq \ell$, set $f_i:=(v,e_i^-)$. 
We claim that for all $1\leq i \leq k$, either $f_i \in D$ or $f_i\in \cA(D)$.

First, let $i <\ell$ and 
suppose for a contradiction that $f:=(w,e_i^+) \in D$. 
Since $D$ is an $\alpha$-structure, $f=(e_{j_1}^+,e_{j_2}^-)$ for some 
$j_1\leq j_2$. 
Now, $j_2 < i$ since otherwise $(e_i^+,e_{j_2}^-)$ is a loop in $D$. 
Furthermore, since $e_{j_1}^+=w$, by definition $e_{j_1} \in Out(w)$, so 
$\ell \leq j_1$ by definition of $\ell$. 
But this implies $\ell< i$, a contradiction. \\
Now let $i \geq \ell$. 
The only additional observation we need to make here is that by 
Proposition \ref{AboveBelow}, $\ell > j$ for all $e_j\in In(v)$. 
The rest is completely analogous to the first case. 
So we can assume that either $(v,e_i^-)\in \cA(D)$, 
or $(v,e_i^-)\in D$. 

We now check that $D'$ is an $\alpha$-structure of rank $r+1$,  
where $S' = \{e_1',\ldots, e_{k+1}'\}$ with 
$$ e_i'=\begin{cases}
	e_i & \text{ if } i<\ell \\
	e & \text{ if } i=\ell \\
	e_{i-1}& \text{ if } i>\ell 
\end{cases}$$
is a set of decisive arcs. That is, 
for all $u,z\in V$ we need to show that $(u,z)\in D'$ if and only if 
$(u,z)= (e_i'^+,e_j'^-)$ for some $1\leq i\leq j\leq k+1$.

Assume first that $(u,z)\in D'$. 
If $(u,z)\in D$ then $(u,z) = (e_i^+,e_j^-)$ for some $1\leq i\leq j\leq k$, and hence $(u,z) = (e_{i_*}'^+,e_{j_*}'^-)$ 
where $i_*\in\{i,i+1\}$ and $j_*\in\{j,j+1\}$. 
If $(u,z)\in D'\sm D$ then either $(u,z)=e=(e_\ell'^+,e_\ell'^-)$ or 
$(u,z)= f_i$ for some $1\leq i \leq t$. 
In the latter, we have that $(u,z)=(e_i^+,w)=(e_i'^+,e_\ell'^-)$ for $i<\ell$, 
and $(u,z)=(v,e_i^-)=(e_\ell'^+, e_{i+1}'^-)$ for $i\geq \ell$. 

For the opposite implication, let first $1\leq i \leq j\leq k+1$ such that $i,j\neq \ell$. 
Note that then $(e_i'^+,e_j'^-)=(e_{i_*}'^+,e_{j_*}'^-)$, 
where $i_*\in\{i,i+1\}$ and $j_*\in\{j,j+1\}$. 
If $i=\ell$, then $j\geq \ell$ and $(e_i'^+,e_j'^-)=e$ or $(e_i'^+,e_j'^-)= f_j$. 
If $j=\ell$, then $i\leq \ell$ and $(e_i'^+,e_j'^-)=e$ or $(e_i'^+,e_j'^-)= f_i$. 
In either case, $(e_i'^+,e_j'^-) \in D'$.
 
Finally, for every existing arc 
$e_i \in S$, we added at most one new arc $f_i$. 
But also, for every vertex $z\in V\setminus\{v,w\}$ at most one of the $f_i$ contains $z$. 
So $\size{\{f_1,\ldots,f_k\}\cap \cA(D+e)}\leq \min \{k, |V|-2\}\leq \min \{r, |V|-2\}$. 
\end{proof}

\subsection{OBreaker's strategy for the monotone rules}
\label{sec:OrientedCycleMain}

\begin{proof}[Proof of Theorem \ref{OrientedCycleMain}]
Recall that OMaker and OBreaker alternately direct edges 
of $K_n$, where OMaker directs exactly one edge in each 
round, and OBreaker directs at least one and at most 
$b$ edges in each round, where $b\geq 5n/6 +2$. OMaker's goal 
is to close a directed cycle, whereas OBreaker's goal is 
to prevent this. 
First, we provide OBreaker with a strategy, 
then we prove that he can follow that strategy 
and that it constitutes a winning strategy. 
At any point during the game let $D$ denote the digraph of 
already chosen arcs. By the rules of the game, $D$ has no 
loops and no reverse arcs. 

The main idea of OBreaker's strategy is to maintain that $D$
consists of a UDB $(A,B)$ and two $\alpha$-structures, located in $V\setminus B$
and $V\setminus A$, respectively, in such a way that each arc of $D$ starts in $A$ or ends in $B$.
For an illustration, see Figure \ref{fig2}. We show shortly that this is enough
to prevent cycles throughout the game. To succeed with a bias
$b\geq 5n/6+2$, we also need to keep control of the size
of the UDB and the rank of the mentioned $\alpha$-structures. For that reason,
we divide the strategy into three stages, in each we maintain different size conditions.
In particular, throughout Stage I, OBreaker builds up two large ``buffer sets'' 
$A'\subseteq A$ and $B'\subseteq B$, until their size is $n-b$. 
These buffer sets then allow OBreaker to play on two --- not necessarily disjoint --- boards in Stages II and III, 
each board having at most $b$ vertices.

 \begin{figure}[bp]
 \centering
\phantom{asdfasdfasdfaasdfasfasdfad}
	 \includegraphics
	 {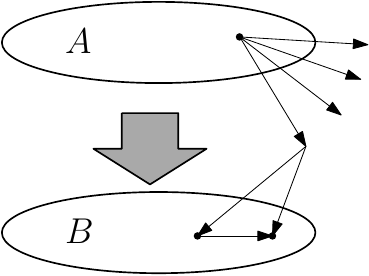}
\caption{The structure of the digraph OBreaker maintains. Thin arcs
represent the decisive arcs of the $\alpha$-structures that live on the vertex sets
$V\setminus A$ and $V\setminus B$ respectively.}
\label{fig2}
\end{figure}

If the strategy asks OBreaker 
to direct an arc $(x,y)$ where $(x,y)\in D$ already, then OBreaker ignores that command and 
proceeds to the next one. If the strategy asks OBreaker 
to direct an arc $(x,y)$ where $(y,x)\in D$ already, then OBreaker forfeits the game. 
By {\em adding $e$ to some $\alpha$-structure in $D$} we mean a strategy 
for OBreaker to direct the arcs $\{f_1,\ldots,f_t\}$ given by Lemma \ref{alphaProc}.

In {\bf Stage I}, OBreaker maintains a UDB $(A,B)$ so that after his 
 move in round $r$  
there exist integers $k,\ell$ such that the following properties hold.
\begin{mytabbing}
\>$(S1.1)$ \> $D(V\setminus B)$ is an $\alpha$-structure of rank $k$,  
	such that $D(V\setminus B)^+ \subseteq A$,\\[5pt]
\>$(S1.2)$ \> $D(V\setminus A)$ is an $\alpha$-structure of rank $\ell$, 
	such that $D(V\setminus A)^- \subseteq B$,\\[5pt]
\>$(S1.3)$ \> $k+\ell = r$, and  \\[5pt]
\>$(S1.4)$ \> $|A|-k = |B|-\ell = r$. 
\end{mytabbing}
OBreaker proceeds to Stage II once $|A|-k = |B|-\ell \geq n/6\geq n-b$, 
that is after round $\left\lceil n/6\right\rceil$.

Initially, before the first move in Stage I, we have $r=k=\ell=0$ and $A=B=\emptyset$. 
Now let $e=(v,w)$ be the arc OMaker directs in a particular round of Stage I. 
Assume first that $e\in \cA(D(V\setminus B))$. 
Then OBreaker adds $e$ to the $\alpha$-structure $D(V\setminus B)$ by directing 
the edges $\{f_1,\ldots,f_t\}\subseteq \cA(D(V\setminus B))$ given by Lemma \ref{alphaProc}.  
If $v\in \rest$ let $x\in V\setminus (A\cup B \cup \{v\})$. Then OBreaker directs all edges $(v,y)$ 
and $(x,y)$ for $y\in B$ 
and sets $A:=A\cup\{v,x\}$. (Note that we might delete a vertex from the
$\alpha$-structure on $V\setminus A$, but by
Lemma \ref{obs:trivialFacts} this does not affect $(S1.2)$.)
Otherwise $v\in A$ already, so OBreaker picks two arbitrary new 
vertices $v',x \in \rest$, directs all edges $(v',y)$ and $(x,y)$ for $y\in B$ 
and sets $A:=A\cup \{v',x\}$. 
In both cases,
he picks an arbitrary element $y' \in V\setminus (A\cup B)$, 
directs all edges $(x',y')$ for $x'\in A$, and sets $B:=B\cup\{y'\}$. 

Assume now that $e\notin \cA(D(V\setminus B))$. Then
since $(A,B)$ is a UDB, $e\in \cA(D(V\setminus A))$. 
Consider the digraph $\Dback$. Note that $(A',B'):=(B,A)$ is a UDB in $\Dback$ 
such that Properties $(S1.1)$--$(S1.4)$ hold (with $k':=\ell$ and $\ell':=k$), 
by Lemma \ref{dualMon}. 
Moreover, $\eback\in \cA(\Dback(V\setminus B'))$. 
Applying the strategy above to $\eback$ and $\Dback$, one obtains arcs 
$f_1,\ldots,f_t$ that OBreaker would direct in the ``dual game", plus updates
of $A'$ and $B'$. 
OBreaker now directs the reversed arcs $\back{f_1},\ldots,\back{f_t}$
and sets $A:=B'$ and $B=A'$.

\noindent
In {\bf Stage II}, OBreaker stops increasing the values $|A|-k$ and $|B| -\ell$. 
Now, he maintains a UDB $(A,B)$ and integers $k,\ell$ such that after each move of OBreaker  
\begin{mytabbing}
\>$(S2.1)$ \> $D(V\setminus B)$ is an $\alpha$-structure of rank $k$,  
	such that $D(V\setminus B)^+ \subseteq A$\\[5pt]
\>$(S2.2)$ \> $D(V\setminus A)$ is an $\alpha$-structure of rank $\ell$, 
	such that $D(V\setminus A)^- \subseteq B$\\[5pt]
\>$(S2.3)$ \>  $|A|-k,\ |B|-\ell \geq n/6\geq n-b$. 
\end{mytabbing}
OBreaker proceeds to Stage III as soon as $A\cup B = V$. 

Again, let $e=(v,w)$ be the arc OMaker directs in a particular round of Stage II 
and assume first that $e\in \cA(D(V\setminus B))$. 
Then OBreaker adds $e$ to the $\alpha$-structure $D(V\setminus B)$ 
using Lemma \ref{alphaProc}.  
If $v\in \rest$ then he directs all edges $(v,y)$ for $y\in B$ 
and sets $A:=A\cup\{v\}$. 
Otherwise $v\in A$ already, so OBreaker picks an arbitrary new 
vertex $v' \in \rest$, directs all edges $(v',y)$ for $y\in B$ 
and sets $A:=A\cup \{v'\}$. \\
Assume now that $e\notin \cA(D(V\setminus B))$. 
Then, since $(A,B)$ is a UDB, $e\in \cA(D(V\setminus A))$.
Similar to Stage I, OBreaker now uses the strategy
for the dual structure and claims the reversed arcs.

\noindent
In {\bf Stage III}, OBreaker maintains a UDB $(A,B)$ with $A\cup B = V$ such that 
\begin{mytabbing}
\>$(S3.1)$ \> $D(A)$ forms an $\alpha$-structure,\\[5pt]
\>$(S3.2)$ \> $D(B)$ forms an $\alpha$-structure, and\\[5pt]
\>$(S3.3)$ \> $|A|,|B| \leq b$. 
\end{mytabbing}
Let again $e=(v,w)$ be the arc OMaker directs in her previous move. 
Either $e\in \cA(D(A))$ or $e\in \cA(D(B))$. 
In the first case, OBreaker adds $e$ to the $\alpha$-structure $D(A)$ 
using Lemma \ref{alphaProc}; 
in the second case, OBreaker adds $e$ to the $\alpha$-structure $D(B)$, 
again by using Lemma \ref{alphaProc}. 
In case $t=0$ in Lemma \ref{alphaProc}, OBreaker does not need to 
direct any arc to reestablish the properties. Then OBreaker directs an arbitrary edge, 
say $e'\in \cA(D(A))$ (or in  $\cA(D(B))$), and adds $e'$ to the $\alpha$-structure 
using Lemma  \ref{alphaProc}.

\medskip
Let us first remark that if OBreaker can follow the proposed strategy 
and reestablish the properties of the certain stage in each move, 
then OMaker can never close a cycle. 
Indeed, throughout the whole game, OBreaker maintains a UDB $(A,B)$ 
such that $D(V\setminus A)$ and $D(V\setminus B)$ form $\alpha$-structures
(cf.~$(S\ast.1)$ and $(S\ast.2)$ of each stage). 
Moreover, also by $(S\ast.1)$ and $(S\ast.2)$ of each stage, 
at any point during the game we have for any $(v,w)\in D$ that 
$v\in A$ or $w\in B$. 
Suppose at some point, OMaker could close a cycle $C$ by directing an 
edge $e=(v,w)$. 
Since $(A,B)$ is a UDB and by the previous comment, 
all edges of $C$ must lie either completely in $V\setminus A$ or 
completely in $V\setminus B$. However, $D(V\setminus A)$ (and $D(V\setminus B)$) 
is an $\alpha$-structure, so by Corollary \ref{cor:noCycles}, 
OMaker cannot close a cycle in $V\setminus A$ (and $V\setminus B$, respectively).

\medskip

It remains to prove that OBreaker can follow the proposed strategy without forfeiting the game, 
that in each round he has to direct at most $b$ edges, 
and that the properties of each stage are reestablished.

Recall that the Properties $(S1.1)$--$(S1.4)$ hold before the first move in Stage I 
(for technical reasons we say ``after round 0") with $r=k=\ell=0$ and $A=B=\emptyset$.

Suppose now that for some $r\geq 0$, after round $r$ in \textbf{Stage I} the Properties $(S1.1)$--$(S1.4)$ hold. 
If $|A|-k,|B|-\ell\geq n/6$, then OBreaker proceeds to Stage II, 
so we can assume $|A|-k=|B|-\ell < n/6$. 
As said previously, since $(A,B)$ is a UDB, all the arcs 
$(x,y)$ with $x\in A$ and $y\in B$ are present in $D$ already, 
so OMaker's arc is completely in $V\setminus A$ or 
completely in $V\setminus B$. 
Assume first that for OMaker's arc $e=(v,w)$ it holds that 
$e\in \cA(D(V\setminus B))$. 
Since $D(V\setminus B)$ forms an $\alpha$-structure by $(S1.1)$, and by 
Lemma \ref{alphaProc}, OBreaker can add $e$ to that 
$\alpha$-structure.
By $(S1.2)$ we have $D(V\setminus A)^- \subseteq B$ before the update of the sets $A$ and $B$. 
So for all $z\in \rest$, all $y\in V\setminus A$, 
either $(z,y)\in D$ or $(z,y)\in \cA(D)$. 
Similarly, by $(S1.1)$ we have $D(V\setminus B)^+ \subseteq A$. 
So for all $z\in \rest$ and all $x\in V\setminus B$ 
either $(x,z)\in D$ or $(x,z)\in \cA(D)$. 
So OBreaker can claim all edges $(v,y)$ (or $(v',y)$) and $(x,y)$ 
for $y\in B$, 
and all edges $(x',y')$ for $x'\in A$ 
as requested by the strategy. \\
By Lemma \ref{alphaProc}, adding $e$ to the $\alpha$-structure in 
$V\setminus B$ takes OBreaker at most $k$ edges to direct.  
Furthermore, since $|A|-k, |B|-\ell$, and $k+\ell$ are bounded by $n/6$
(by assumption and $(S1.4)$), 
the strategy asks OBreaker to direct at most 
$$k + 2|B| +|A|+2 = 2(|B|-\ell) +2 (k+\ell) + (|A|-k)+2 
	\leq 5 \left\lfloor\frac{n}{6}\right\rfloor +2 
	\leq b$$
edges in one round of Stage I. 
We need to show that the properties are restored. 
For ease of notation, let us assume that $v$ was in $\rest$, 
so OBreaker added $v$ (and not $v'$) to $A$. 
Let $f_1,\ldots,f_t$ be the arcs 
OBreaker directs in round $r+1$. 
Let $D'=D\cup\{e,f_1,\ldots, f_t\}$ be the new digraph 
after OBreaker's move. 
It is obvious from the strategy description that after the update the pair 
$(A,B)$ forms a UDB again.
In round $r+1$, OBreaker adds $v$ and some vertex $x\in\rest$ to $A$, 
and some vertex $y'\in \rest$ to $B$. Therefore, 
by Lemma \ref{obs:trivialFacts}, and by Lemma \ref{alphaProc}, 
$D'(V\setminus B)$ is then an $\alpha$-structure 
of rank $k+1$, and $D'(V\setminus B)^+\subseteq A$. 
So $(S1.1)$ holds again. 
OBreaker's arcs belong to $A\times V$, so OBreaker may delete vertices 
from the $\alpha$-structure $D(V\setminus A)$.  
Hence, $(S1.2)$ holds by Lemma \ref{obs:trivialFacts}. 
For $(S1.3)$ note that after OBreaker's move we have one $\alpha$-structure
of rank $k+1$ and one of rank $\ell$.
Finally, for $(S1.4)$ note that the rank of the first $\alpha$-structure (with rank $k$) increases by one, while 
OBreaker adds two vertices to $A$. 
For the second $\alpha$-structure we do not change the rank, 
while we increase $|B|$ by one. 
\noindent
Assume now that $e\notin \cA(D(V\setminus B))$, that is $e\in \cA(D(V\setminus A))$. 
Applying the previous argument to the dual $\Dback$, it is clear that OBreaker can follow that 
strategy and that this strategy restablishes the Properties $(S1.1)$--$(S1.4)$ for $D$, using the 
self-duality of those properties and of $\alpha$-structures (cf.~Proposition \ref{dualMon}).

\medskip

For \textbf{Stage II}, $(S1.1)$ and $(S1.2)$ trivially imply $(S2.1)$ and $(S2.2)$. 
$(S2.3)$ follows by assumption of entering Stage II. 
So assume the three properties hold before OMaker's move in this stage. 
Assume first that for OMaker's arc $e=(v,w)$ it holds that 
$e\in \cA(D(V\setminus B))$.
As in Stage I, 
OBreaker can add $e$ to the $\alpha$-structure in 
$V\setminus B$.  
Similarly as in Stage I, by $(S2.2)$, 
for all vertices $z\in \rest$, all $y\in B$, either
$(z,y)\in D$ or $(z,y)\in \cA(D)$. 
So OBreaker can direct all edges $(v,y)$, or $(v',y)$, respectively, 
for $y\in B$ as requested by the strategy. \\
Furthermore, the strategy asks him to direct at most 
$$ |B|+k = |V|-(|A|-k) \leq b$$ 
edges, by Property $(S2.3)$. 
Finally, we verify that the properties are restored. 
By Lemma \ref{obs:trivialFacts} and Lemma \ref{alphaProc}, $D(V\setminus B)$ is again an $\alpha$-structure. 
Since $v$ is added to $A$ by directing all 
edges $(v,y)$ for $y\in B$, $(S2.1)$ follows. 
Again, $(S2.2)$ follows from Lemma \ref{obs:trivialFacts}. 
For $(S2.3)$, note that only the rank of the $\alpha$-structure $D(V\setminus B)$ increased by one, 
and OBreaker added exactly one new vertex to $A$ ($v$ or $v'$). 
The case $e\in \cA(D(V\setminus A))$ is analogous due to the duality  
of the properties.

\medskip 
Finally, it is straight-forward that OBreaker can follow the 
strategy proposed in {\bf Stage III}.
Since OBreaker plays in Stage II until $A\cup B = V$ 
(and the sets $A,B$ indeed grow in each round), and 
by $(S2.3)$ it follows that $|A|,|B|\leq b$. 
He then plays either inside $A$ or $B$ according to the strategy 
given by Lemma \ref{alphaProc} until all edges are directed. 
Therefore, in one round, OBreaker needs to direct at most 
$|A| -2< b$ or $|B|-2< b$ edges to add $e$ (or $e'$ respectively) 
to the $\alpha$-structure.

\medskip
This finishes the proof of Theorem \ref{OrientedCycleMain}.
\end{proof}

\section{The Oriented-cycle game -- strict rules} \label{strict}

In this section, we give a strategy for OBreaker for the strict Oriented-cycle game when playing with bias $b$. 
We prove that this strategy constitutes a winning strategy if $19n/20\leq b\leq n-3$. 
For $b\geq n-2$, a winning strategy is already given by the trivial strategy mentioned in the introduction. 
 
In our strategy for the monotone rules, OBreaker aims for a
UDB $(A,B)$ as global structure, and handles OMaker's edges locally using 
$\alpha$-structures.
The reason for $\alpha$-structures being powerful is that OBreaker stops directing 
further edges in this game once the digraph has the desired local structure. 
In the strict rules, this is of course no longer possible. 
Therefore, we introduce different 
structures  
that are more robust to adding edges. 
The strategy for the strict rules consists of two stages. 
Similarly as in the monotone game, in Stage I, OBreaker creates two large buffer sets 
$A'$ and $B'$, both being independent in $D$, such that $(A',B')$ forms a UDB in $D$. 
In Stage II, OBreaker then maintains two such sets of a certain minimum size.  

In Subsection~\ref{sec:structures}, we define the structures that OBreaker aims to maintain 
and state lemmas (cf.~Lemmas \ref{lem:BaseI}, \ref{lem:AddEdgesI}, \ref{lem:Transition}, \ref{lem:BaseII}, \ref{lem:AddEdgesII}) 
that are necessary to prove that OBreaker can win. 
We defer the proofs of these lemmas to Sections~\ref{sec:StrictStageI} and~\ref{sec:StrictStageII}. 
In Subsection~\ref{sec:OrientedCycleMain2}, we describe the strategy of OBreaker precisely and 
prove that it constitutes a winning strategy.  

\subsection{Riskless and protected digraphs}\label{sec:structures}

We call the structure of $D$ that OBreaker 
maintains during Stage I {\em riskless} and define it 
as follows. Following general notation, a {\em down set} of $[n]$ is a set $[k]$, for some $k \leq n$, 
and an {\em upset} of $[n]$ is a set $[n]\setminus [k]$, for some $k \leq n$. 

\begin{definition}\label{def:riskless}
A digraph $D$ is called {\em riskless of rank $r$} if 
there is a UDB $(A,B)$ with partitions $A=A_S\cup A_0$ and $B=B_S\cup B_0$
such that the following properties hold:
 
\begin{enumerate}
\item[$(R1)$] {\em Size:} $\size{\size{A}-\size{B}}\leq 1$  
and $|A_S|=|B_S|=r$. 
\item[$(R2)$] {\em Structure of $A_S$ and $B_S$:}
There exist enumerations $A_S=\{v_1,\ldots v_r\}$ and \\ $B_S=\{w_1,\ldots w_r\}$ 
such that
\begin{itemize}
		\item[$(R2.1)$] $(v_1,\ldots,v_r)$ and $(w_1,\ldots,w_r)$ induce transitive tournaments in $D$. 
		\item[$(R2.2)$] For all $z\in A_0\cup\rest$: $\{i: (v_i,z)\in D\}$ is a down set of $[r]$. 
		\item[$(R2.3)$] For all $z\in B_0\cup\rest$: $\{i: (z,w_i)\in D\}$ is an upset of $[r].$ 
\end{itemize}
\item[$(R3)$] {\em Stars:} For every $1\leq i\leq r$:
\begin{itemize}
		\item[$(R3.1)$] $e_D(v_i,A_0)\leq r+1-i$ and $e_D(v_i,\rest)\leq \max\{\size{A},\size{B}\}$.
 		\item[$(R3.2)$] $e_D(B_0,w_i)\leq i$ and $e_D(\rest,w_i)\leq \max\{\size{A},\size{B}\}$.  
\end{itemize}

\item[$(R4)$] {\em Edge set:} $D=D(A,B)\cup D(A_S,V\setminus B)\cup D(V\setminus A,B_S).$
\end{enumerate}
\end{definition}
\begin{figure} [bp]
\centering
\includegraphics[scale=0.9]{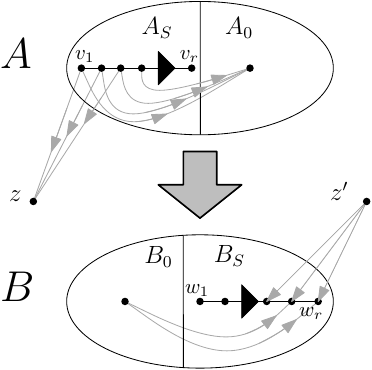}
 \caption{The structure of a {\em riskless} digraph.}
 \label{fig:riskless1}
 \end{figure}
 An illustration of some of the properties can be found in Figure \ref{fig:riskless1}. 
Observe that the definition of riskless is self-dual in the following sense. 
Recall that $\Dback$ denotes the digraph obtained by reversing all arcs of a digraph $D$. 
\begin{observation}\label{dual}
A digraph $D$ is a riskless digraph of rank $r$ with UDB $(A,B)$ 
if and only if $\Dback$ is a riskless digraph of rank $r$ with UDB $(B,A)$.  
\end{observation}

Note that the empty graph is riskless of rank $0$. 
In the strict Oriented-cycle game, assume that after round $r$, 
for some $0\leq r$, 
the digraph $D$ of directed edges is riskless of rank $r$ with UDB~$(A,B)$. 
Let $e\in \cA(D)$ be the arc that OMaker directs in round $r+1$. 
OBreaker's move is divided into what we call a {\em base strategy} 
and an {\em add-edges strategy}. In the base strategy, 
OBreaker chooses {\em at most} $b$ arcs to restore the properties of a riskless digraph. 
In the add-edges strategy, he directs further edges 
keeping the digraph riskless so that the total number of edges he directs in one 
round is exactly $b$. 
The following two lemmas make sure that he  can indeed do so for a certain number of rounds. 
 
\begin{lemma}[Base strategy for Stage I]\label{lem:BaseI}
Let $19n/20\leq b\leq n-3$.
For a non-negative integer $r\leq n/25-1$, 
let $D$ be a digraph which is riskless of rank $r$ with UDB $(A,B)$ as given 
in Definition \ref{def:riskless}. 
Assume that $\size{D}=r(b+1)$. 
Let $e\in \cA(D(V\setminus B))$ be an available arc in $V\setminus B$. 
Then there exists a set  $\{f_1,\ldots,f_t\} \se \cA\left(D+e\right)$ of 
at most $b$ available arcs  
such that $D':=D\cup\{e,f_1,\ldots,f_t\}$ is a riskless digraph of rank $r+1$.
Moreover, $e_{D'}(\restStrich,w_1')=0$, where $(A',B')$ is the underlying UDB of $D'$,
with $B'=B_S'\cup B_0'$ and $B_S'=\{w_1',\ldots, w_{r+1}'\}$ as 
in Definition \ref{def:riskless}.
\end{lemma}
The property $e_{D'}(\restStrich,w_1')=0$ in this lemma is 
needed to apply the next lemma. 

\begin{lemma}[Add-edges strategy for Stage I]\label{lem:AddEdgesI}
Let $19n/20\leq b\leq n-3$.
For a non-negative integer $r\leq n/25-1$, 
let $D$ be a digraph which is riskless of rank $r+1$ with UDB $(A,B)$ as given 
in Definition \ref{def:riskless}. 
Assume that $r(b+1)\leq \size{D}\leq (r+1)(b+1)<\binom{n}{2}$. 
Let $w_1$ be the source of the tournament inside $B$, as given in Property $(R2)$, 
and assume that $e_D(\rest,w_1)=0$ holds. 
Then there is a set of $(r+1)(b+1)- \size{D}\leq b$ available arcs $\cF \subseteq \cA(D)$ 
such that $D':=D\cup\cF$ is a riskless digraph of rank $r+1$. 
\end{lemma}

The structure that OBreaker aims to maintain in Stage II is similar to the one given for Stage I.
The most important difference is that in this new structure we partition the sets of the UDB
further to distinguish the vertices according to their chance to become part
of a directed cycle. We maintain partitions
$A=A_D\cup A_{AD}\cup A_{S}\cup A_{0}$ 
and $B=B_D\cup B_{AD}\cup B_{S}\cup B_{0}$.
(The subscripts stand for {\em Dead}, {\em Almost Dead} and {\em Star}.)
$A_0$ and $B_0$ form the aforementioned buffer sets together with all dead vertices.

\begin{definition}\label{def:protected}
A digraph $D$ on $n$ vertices is called {\em protected} if 
there is a UDB $(A,B)$ with partitions $A=A_D\cup A_{AD}\cup A_{S}\cup A_{0}$ 
and $B=B_D\cup B_{AD}\cup B_{S}\cup B_{0}$ 
such that the following properties hold:
 
\begin{enumerate}
\item[$(P1)$] {\em Sizes:} $|A|,|B|\geq n/10 + 1$, and $|A_D\cup A_0|,\ |B_D\cup B_0|\geq n/10$.

\item[$(P2)$] {\em Dead vertices:} The pairs $(A_D,V\setminus A_D)$ and $(V\setminus B_D,B_D)$ form $UDBs$,
	$D(A_D)$ and $D(B_D)$ are transitive tournaments.

\item[$(P3)$] {\em Almost-dead vertices:} The pairs $(A_{AD},V\setminus A)$ and $(V\setminus B,B_{AD})$
	form $UDBs$.

\item[$(P4)$] {\em Structure of $A_{AD}\cup A_S$ and $B_{AD}\cup B_S$:} 
	There exist integers $k_1,\ell_1\geq 0$ and $0\leq k_2,\ell_2\leq n/25$
	and enumerations
\begin{align*}
A_{AD}=\{v_1,\ldots, v_{k_1}\}, & \  A_{S}=\{v_{k_1+1},\ldots, v_{k_1+k_2}\} \text{ and}\\
B_{S}=\{w_1,\ldots, w_{\ell_2}\}, & \ B_{AD}=\{w_{\ell_2+1},\ldots, w_{\ell_2+\ell_1}\}  \text{ such that}
\end{align*}
\begin{itemize}
	\item[$(P4.1)$] $(v_1,\ldots,v_{k_1+k_2})$ and $(w_1,\ldots,w_{\ell_1+\ell_2})$
		induce transitive tournaments.
	\item[$(P4.2)$]	 For all $z\in A_0\cup\rest$: 
					$\{i: (v_i,z)\in D\}$ is a down set of $[k_1+k_2]$
	\item[$(P4.3)$]	 For all $z\in B_0\cup\rest$: 
					$\{i: (z,w_i)\in D\}$ is an upset of $[\ell_1+\ell_2].$
\end{itemize}

\item[$(P5)$] {\em Stars:}
\begin{itemize}
	\item[$(P5.1)$] For all $1\leq i\leq k_2$: $e(v_{k_1+i},A_0)\leq n/25+1 -i$. 
	\item[$(P5.2)$] For all $1\leq i\leq \ell_2$: $e(B_0,w_{i})\leq n/25 - \ell_2+i$. 
\end{itemize}
\item[$(P6)$] {\em Edge set:} 
$E(D)=D(A,B)\cup D(A\setminus A_0,V\setminus B)\cup D(V\setminus A,B\setminus B_0).$
\end{enumerate}
\end{definition}
%

 \begin{figure}[tb]
 \centering
\begin{subfigure}{0.48\textwidth}
	\centering
	 \includegraphics[scale=0.5]
	 {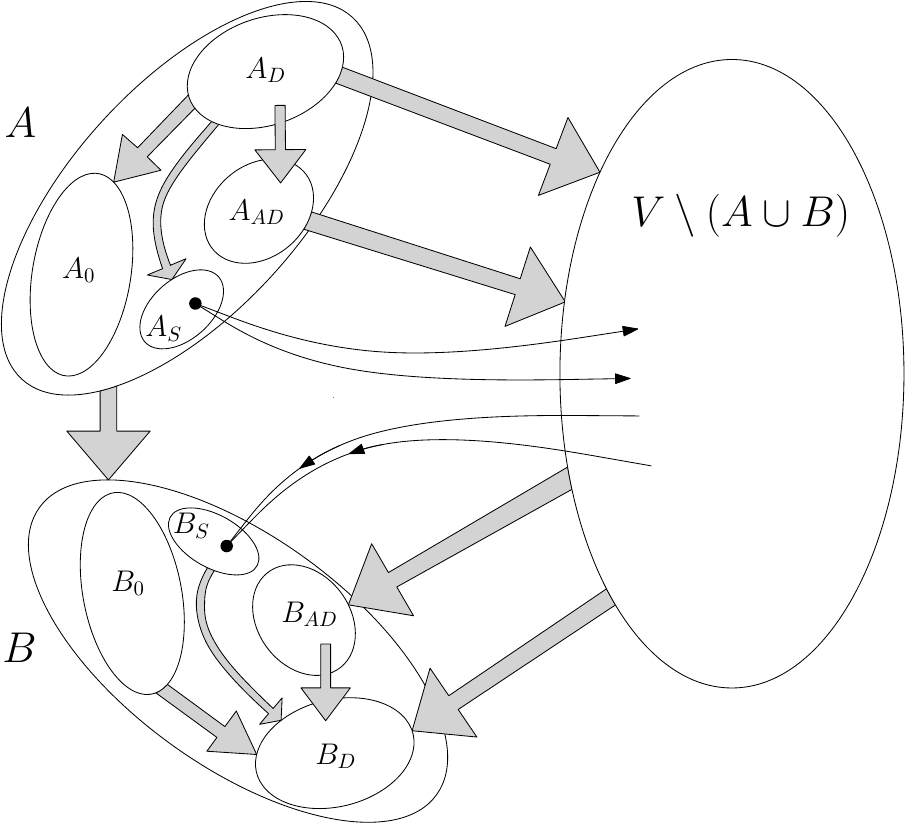}
\end{subfigure} 
\quad
 \begin{subfigure}{0.48\textwidth}
 	\centering
	\includegraphics[scale=0.5]
	{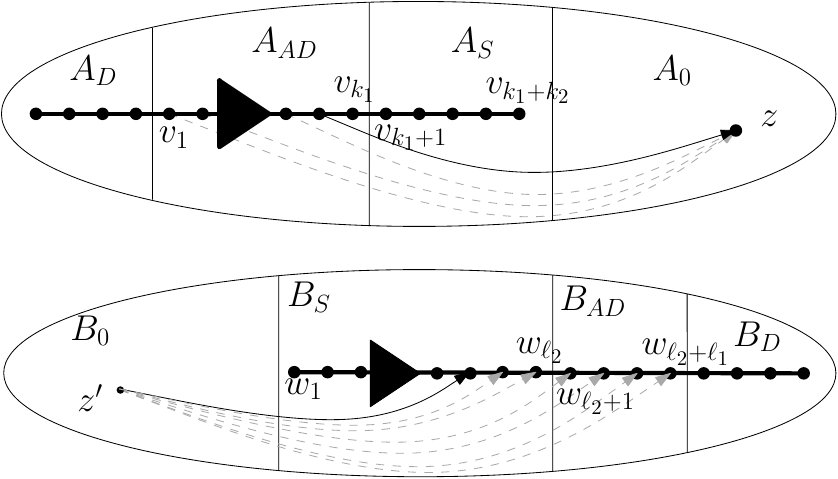}
 \end{subfigure}
 \caption{The global and local structure of a protected digraph. 
}
 \label{fig_small}
 \end{figure}
An illustration of a protected digraph can be found in Figure \ref{fig_small}.

Again, the definition of protected is self-dual in the following sense. 
\begin{observation}\label{dual2}
A digraph $D$ is a protected digraph with UDB $(A,B)$ 
if and only if $\Dback$ is a protected digraph with UDB $(B,A)$.  
\end{observation}

When proceeding from Stage I to Stage II, we need to ensure that
OBreaker can obtain a protected digraph in the very first round of Stage II.
For this, we prove the following lemma, which roughly states that a riskless digraph
with large enough rank is protected.

\begin{lemma}\label{lem:Transition}
Let $n$ be large enough, and let $19n/20\leq b\leq n-3$.
Let $D$ be a digraph on $n$ vertices which is riskless of rank $r=\left\lfloor n/25\right\rfloor$, 
and assume that $|D| = r(b+1)$. Then $D$ is protected. 
\end{lemma}

During Stage II, OBreaker's strategy is similar to that of Stage I.
After OMaker adds one further arc $e$ to the digraph $D$, OBreaker's goal is as follows. At first he follows what we call a {\em base strategy} in which he restores the properties of a protected digraph, and afterwards he follows 
what we call an {\em add-edges strategy} in which he keeps the digraph protected until exactly $b$ arcs are chosen. The following two lemmas make sure that he indeed can do so until the end of the game.

\begin{lemma}[Base strategy for Stage II]\label{lem:BaseII}
Let $D$ be a digraph which is protected and let $e=(v,w)\in \cA(D(V\setminus B))$ be an 
available arc in $V\setminus B$. Then there exists 
a set $\{f_1,\ldots,f_t\} \se \cA(D+e)$ 
of at most $b$ available arcs 
such that $D':=D\cup\{e,f_1,\ldots,f_t\}$ is protected. 
\end{lemma}

\begin{lemma}[Add-edges strategy for Stage II]\label{lem:AddEdgesII}
Let $D$ be a digraph which is protected. 
Then, unless $D$ is a tournament on $n$ vertices, 
there exists an available arc $f\in \cA(D)$ such that 
$D+f$ is protected.  
\end{lemma}
The choice of the families $\{f_1,\ldots,f_t\}$ in Lemma \ref{lem:BaseI} and \ref{lem:BaseII}
strongly depends on the edge $e$ which represents the arc chosen by OMaker. 
We postpone the proofs of 
Lemmas \ref{lem:BaseI}, \ref{lem:AddEdgesI}, \ref{lem:Transition}, \ref{lem:BaseII} and \ref{lem:AddEdgesII} 
to later sections in order to keep
the strategy description rather short.
We prove Lemma \ref{lem:BaseI} and Lemma \ref{lem:AddEdgesI} in Section~\ref{sec:StrictStageI}; 
Lemma \ref{lem:Transition}, Lemma \ref{lem:BaseII} and Lemma \ref{lem:AddEdgesII}
are proved in Section \ref{sec:StrictStageII}.

\subsection{OBreaker's strategy for the strict rules}
\label{sec:OrientedCycleMain2}

In the following we descibe OBreaker's strategy for the strict rules. Given that Lemmas \ref{lem:BaseI}, \ref{lem:AddEdgesI}, \ref{lem:Transition}, \ref{lem:BaseII} and \ref{lem:AddEdgesII} hold, we also show that this strategy constitutes a winning strategy. 
Whenever OBreaker cannot follow the proposed strategy he forfeits the game. 

{\large {\bf Stage I}} lasts exactly $\left\lfloor n/25 \right\rfloor$ rounds. OBreaker always ensures that right after his move the digraph $D$ of all directed edges is riskless. 

For some $0\leq r\leq n/25-1$, after round $r$, assume $D$ is riskless of rank $r$ with UDB $(A,B)$. 
Let $e=(v,w)$ be the arc OMaker directs in round $r+1$. 

Suppose first that $e\in \cA(D(V\setminus B))$. OBreaker then 
follows the {\em base strategy} and chooses 
$t\leq b$ arcs given by Lemma \ref{lem:BaseI}. He then 
chooses $b-t$ further arcs according to Lemma \ref{lem:AddEdgesI}. 
That is, in total OBreaker chooses a set $\cF$ of exactly $b$ available arcs
such that $D \cup \{e\} \cup \cF$ is riskless of rank $r+1$.

Assume then that $e=(v,w)\not\in \cA(D(V\setminus B))$ and thus
$e\in \cA(D(V\setminus A))$, since $(A,B)$ forms a UDB. 
By Observation \ref{dual}, $\Dback$ is also riskless of rank $r$ with UDB $(B,A)$. 
Applying Lemma \ref{lem:BaseI} and Lemma \ref{lem:AddEdgesI}
OBreaker then finds a set $\cF$ of exactly $b$ available arcs such that
$\Dback \cup \{\eback\} \cup \cF$ is riskless of rank $r+1$.
OBreaker then chooses the $b$ arcs $\back{\cF}$.

{\bf Stage II} starts in round $\left\lfloor n/25\right\rfloor+1$. OBreaker now ensures
that right after his move the digraph $D$ is protected.

For some $r \geq \left\lfloor n/25 \right\rfloor$, after round $r$, assume $D$ is protected with UDB $(A,B)$.  Let $e=(v,w)$ be the arc OMaker directs in round $r+1$. 

Again, suppose first that $e\in \cA(D(V\setminus B))$. 
OBreaker then follows the {\em base strategy} and chooses 
$t\leq b$ arcs given by Lemma \ref{lem:BaseII}. Afterwards, he
chooses $b-t$ further arcs (unless there are fewer than $b-t$ available edges 
whence he directs all remaining edges), by applying Lemma \ref{lem:AddEdgesII} iteratively.
So, in total OBreaker chooses a set $\cF$ of exactly $b$ 
(or at most $b$ in the last round of the game)
available arcs such that $D \cup \{e\} \cup \cF$ is protected.

Assume then that $e=(v,w)\not\in \cA(D(V\setminus B))$ and thus
$e\in \cA(D(V\setminus A))$, since $(A,B)$ forms a UDB. 
By Observation \ref{dual2}, $\Dback$ is also protected with UDB $(B,A)$. 
Applying Lemma \ref{lem:BaseII} and Lemma \ref{lem:AddEdgesII}
OBreaker then finds a set $\cF$ of exactly $b$ (or at most $b$ in the last round of the game) available arcs such that $\Dback \cup \{\eback\} \cup \cF$ is protected.
OBreaker then chooses the arcs $\back{\cF}$.

\medskip

We finish this section with the proof that the proposed strategy constitutes a winning strategy. 
The empty graph is riskless of rank 0. Assume now that, for $0\leq r \leq \left\lfloor n/25\right\rfloor$, 
the digraph $D$ is riskless of rank $r$ after round $r$. Let $e$ be the arc OMaker directs in round $r+1$. 
Then, by Lemma \ref{lem:BaseI} and Lemma \ref{lem:AddEdgesI}, 
OBreaker can direct exactly $b$ edges $\cF$ such that the digraph $D\cup\cF\cup\{e\}$ is riskless of 
rank $r+1$. In case OMaker's edge satisfies $e=(v,w)\not\in \cA(D(V\setminus B))$, just notice that
$D \cup \{e\} \cup \back{\cF}$ is riskless  after OBreaker's move, by the choice of $\cF$
and by Observation~\ref{dual}. 
Furthermore, after round $r= \left\lfloor n/25\right\rfloor$, 
$D$ is protected by Lemma~\ref{lem:Transition}. 
Then, in Stage~II, Lemma~\ref{lem:BaseII} and Lemma \ref{lem:AddEdgesII} guarantee analogously that the 
digraph $D$ is protected after each move of OBreaker and thus OBreaker always can follow that part of his strategy as well. 
To finish the discussion, it therefore remains to show that a protected digraph is always acyclic, and thus OMaker could not have closed a cycle in her move.

\begin{lemma}
If $D$ is a protected digraph, then $D$ is acyclic.
\end{lemma}
\begin{proof}
Let $D$ be a protected digraph with UDB $(A,B)$,  
and suppose there is a directed cycle $C$ in $D$. 
By Property $(P6)$, for each $(v,w)\in D$, we have
$v\in A$ or $w\in B$. Therefore, the underlying edges of $C$
either only contain vertices from $A$ or
only contain vertices from $B$. By Observation \ref{dual2},
we may assume without loss of generality that $C\subseteq D(A)$.
Again by Property $(P6)$, $C$ must use only vertices from $A\setminus A_0$.
However, by Properties $(P2)$ and $(P4.1)$, $A\setminus A_0$ induces a transitive tournament,
and thus does not contain a directed cycle, a contradiction.
\end{proof}

\section{Strict rules - Stage I}\label{sec:StrictStageI}

In the following we prove Lemma \ref{lem:BaseI} and Lemma \ref{lem:AddEdgesI}.
First, let us pin down a proposition which we shall use frequently. 

\begin{proposition}\label{obs:trivial}
Let $n$ be large enough, and let $b\leq n-3$ 
and $r\leq n/25$. Let 
$D$ be a riskless digraph of rank $r$, with underlying UDB $(A,B)$, such that $|D|\leq r(b+1)$.
Then 
\begin{enumerate}[$(i)$]
\item \label{eq:sizeA} $|A|, |B| < n/5+1$ and $\size{\rest}> 3n/5-1$.
\item \label{eq:sizeGood} $\size{X_A}, \size{Y_B} > 2n/5-2$ where $X_A = \{z\in \rest: e_D(A,z)= 0\}$ 
	and $Y_B = \{z\in \rest: e_D(z,B)= 0\}$.
\end{enumerate}
\end{proposition}

\begin{proof}
Note first that since $(A,B)$ is a UDB 
$|A|\cdot |B| \leq |D|\leq r(b+1)< n^2/25$, 
by assumption on $r$ and $b$. 
Since $\size{|A|- |B|}\leq 1$ by Property $(R1)$, it follows that 
$\max\{|A|,|B|\} < n/5+1$, and that 
$|A|+|B|< 2n/5+1$. 
Therefore, $\size{\rest}> 3n/5-1$. 
Let $\widebar{X_A}:= \{z\in \rest: e_D(A,z)> 0\}$. 
Then, by Properties $(R4)$ and $(R2.2)$, 
\begin{align*}
\widebar{X_A} &=\left\{ z\in \rest : (x,z)\in D \text{ for some } x\in A \right\}\\
	&= \left\{ z\in \rest : (x,z)\in D \text{ for some } x\in A_S \right\}\\
	&=\left\{ z\in \rest : (v_1,z)\in D \right\}.
\end{align*}
So, by Property $(R3.1)$ and \eqref{eq:sizeA},
$|\widebar{X_A}| \leq \max\{|A|,|B|\}<n/5+1,$ 
and hence, 
$$\size{X_A}=\size{\rest}-|\widebar{X_A}|>\frac{2n}{5}-2.$$
Similarly,  $\size{Y_B}>2n/5-2$.
\end{proof}

Now, we prove Lemma \ref{lem:BaseI}.
It ensures that OBreaker has a strategy to reestablish the properties
of a riskless graph throughout Stage I.

\begin{proof}[Proof of Lemma \ref{lem:BaseI}] 
For $e\in \cA(D(V\setminus B))$ given by the lemma, let $v=e^+$ and $w=e^-$.
At first, let us fix distinct vertices $x_A, y_B\in V\setminus (A\cup B\cup \{v,w\})$ 
with $e_D(A,x_A)=0$ and $e_D(y_B,B)=0$. Note that Proposition \ref{obs:trivial}
guarantees their existence since $r\leq n/25$ and $b\leq n-3$. 
Define $u_S\in A_0\cup \rest$ and $u_A\in \rest$ by 
\begin{align*} 
u_S :=
\begin{cases}
x_A 	& \text{if }v\in A_S\\
v 	& \text{if }v \notin A_S
\end{cases}
\qquad\text{and}\qquad
u_A :=
\begin{cases}
x_A 	& \text{if }v\in A\\
v 	& \text{if }v \notin A  .
\end{cases}
\end{align*}
Our goal is to add $u_S$ to the set $A_S$ of star centers, $u_A$ to $A$ and $y_B$ to $B$.
Note that the two vertices $u_s$ and $u_A$ are equal unless $v\in A_0$. Moreover, let
\begin{align*}
\ell :=
\begin{cases}
\min \{i:\ (v_i,v)\notin D\} & \text {if minimum exists}\\
r+1 & \text{otherwise}
\end{cases}
\end{align*}
and observe that $v_{\ell}=v$ if $v\in A_S$ 
(by Property $(R2.1)$). Set
\begin{align*} 
v_i' :=
\begin{cases}
v_i 	& 1\leq i\leq \ell -1\\
u_S 	& i=\ell\\
v_{i-1}	& \ell +1\leq i\leq r+1\\
\end{cases}
\qquad\text{and}\qquad 
w_i' :=
\begin{cases}
y_B 	& i=1\\
w_{i-1} 	& 2\leq i\leq r+1.\\
\end{cases}
\end{align*}
These vertices are used to form the new centers of the stars in $A_S$ and $B_S$. Now, 
choose $\{f_1,\ldots,f_t\}$ to be $\{f_1,\ldots,f_t\} = (\cF_1\cup\cdots\cup \cF_7)\cap \cA(D)$, where 
\begin{align*}
	\cF_1 &:= \big\{(u_A,y) : y \in B\big\}\\
	\cF_2 &:= \big\{(x,w_1') : x\in A\cup \{u_A\}\big\}\\
	\cF_3 &:= \big\{(w_1',w_i') : 2\leq i\leq r+1\big\}\\
	\cF_4 &:= \big\{(v_i',w) : 1\leq i\leq \ell \big\}\\
	\cF_5 &:= \big\{(v_i',v_{\ell}') : 1\leq i\leq \ell -1  \big\}\\
	\cF_6 &:= \big\{(v_{\ell }',v_i') : \ell+1 \leq i\leq r+1 \big\}\\
	\cF_7 &:= 
	\big\{(v_{\ell}',z) : z\in\rest\cup A_0 \text{ and }(v_{\ell +1}',z)\in D \big\}, 
\end{align*}
where we use the convention that $\cF_7=\emptyset$ if $\ell = r+1$ 
(and thus $v_{\ell +1}'$ does not exist). 
The illustration of these arc sets can be found in Figure \ref{fig10}.

 \begin{figure}[b]
 \centering
 \begin{subfigure}{0.55\textwidth}
	\centering
	 \includegraphics[scale=0.6]
	 {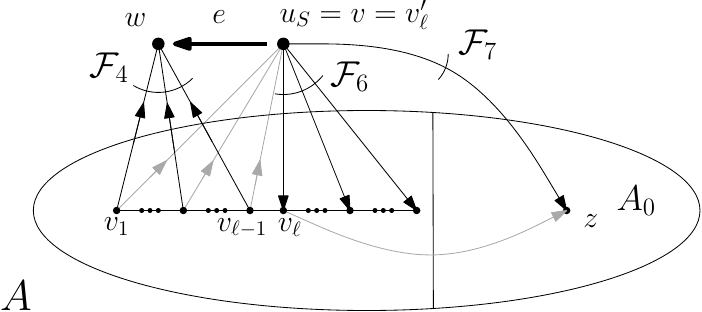}
\end{subfigure} 
\quad
 \begin{subfigure}{0.4\textwidth}
 	\centering
	\includegraphics[scale=0.7]
	{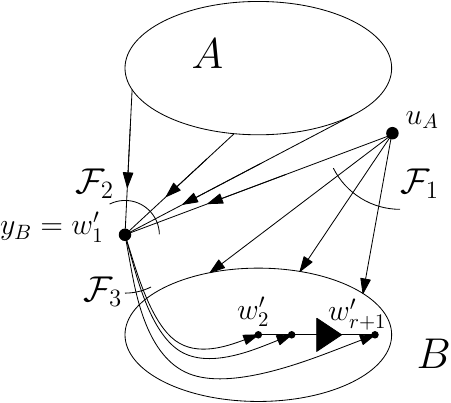}
 \end{subfigure}
 \caption{The case $v\in \rest$. $\cF_4, \cF_6, \cF_7$ on the left include vertex $v$ into the tournament in $A$ and restore $(R2.2)$. 
 $\cF_1, \cF_2, \cF_3$ on the right restore the UDB-property and $(R2.1)$, 
 after $u_A$ and $y_B$ are added to $A$ and the tournament in $B$ respectively.}
 \label{fig10}
 \end{figure}

To show that this choice of arcs is suitable for the lemma, 
we first show that $f_i\neq \back{e}$ for all $1\leq i\leq t$, 
that $\cF_1\cup\cdots\cup\cF_7 \subseteq D\cup \cA(D)$, 
and that $t\leq b$. 
Note that $\cF_i \subseteq D\cup \cA(D)$ implies that  $\cF_i\subseteq D'$, where $D'=D\cup \{e,f_1,\ldots,f_t\}$. 
We use this information to show that $D'$ is riskless of rank $r+1$
with some UDB $(A',B')$. Finally, we show that $e_{D'}(\restStrich,w_1')=0$, 
where $w_1'$ is the source of the tournament in $B'$, given by Property $(R2.1)$.

For the first part, $\eback =(w,v)\notin \cF_1$ since by assumption 
$v \notin B$, $\eback\notin \cF_2\cup \cF_3$ by choice of $w_1'=y_B$, 
and $\eback\notin \cF_4$ since $v\neq w$.
Assume now that $(w,v)=(v_i',v_{\ell}')\in \cF_5$ for some $1\leq i\leq \ell -1$.
Then, $(w,v)=(v_i,v)\in D$ by definition of $\ell$, 
a contradiction to $e=(v,w)\in \cA(D(V\setminus B))$. 
So, $\eback \notin \cF_5$. 
For $\cF_6$, assume that $(w,v)=(v_{\ell }',v_i')\in \cF_6$ for some $\ell+1 \leq i\leq r+1$.    
Then $v\neq v_\ell'$, so $v_\ell' = u_S=x_A\neq w$ by definition of $x_A$,
a contradiction. 
So, $\eback \notin \cF_6$. 
Finally, if $(w,v)\in \cF_7$ then $v\in\rest\cup A_0$, so 
$w=v_\ell'=u_S=v$ by definition of $u_S$, 
a contradiction since $(v,w)\in \cA(D)$ and thus, it is not a loop.   
Hence, $\eback \notin \cF_7$, as well.

To see that $\cF_1\subseteq D\cup \cA(D)$ 
note that $(y,u_A)\notin D$ for all $y\in B$, since 
$u_A\in \rest$ and by Property $(R4)$. 
Similarly, $\cF_2\cup \cF_3\subseteq D\cup \cA(D)$ 
since $w_1'=y_B\in \rest$ and by Property $(R4)$. 
Assume now that $(w,v_i')\in D$ for some $1\leq i\leq \ell$.
Then $w=v_j\in A_S$ for some $1\leq j\leq r$, by Property $(R4)$.
Moreover, $v_i'\neq v$, since $e=(v,w)\in \cA(D(V\setminus B))$,
and $v_i'\neq x_A$ by choice of $x_A$ and Property $(R4)$.
This yields $v_i'\neq u_S$ and therefore $i\neq \ell $. But then $v_i'=v_i$, and $j<i\leq \ell -1$ by Property $(R2.1)$.
By definition of $\ell$ we conclude $(w,v)=(v_j,v)\in D$,
in contradiction to $e\in \cA(D(V\setminus B)).$ 
Thus, $\cF_4 \subseteq D\cup \cA(D)$.
For $\cF_5$, we note that $(v_\ell',v_i')=(u_S,v_i)\notin D$ for all $1\leq i<\ell$  
by Property $(R4)$ since $u_S\notin A_S$ and $v_i\notin B$.
Now, assume that $(v_i',v_{\ell}')=(v_{i-1},u_S)\in D$ for some $\ell+1\leq i\leq r+1$.
Then $u_S\neq x_A$ by the choice of $x_A$. It follows $(v_{i-1},v)=(v_{i-1},u_S)\in D$. 
But then, by definition of $\ell$ and $(R2.2)$, we obtain 
$i-1<\ell$, a contradiction. Thus $\cF_6 \subseteq D\cup \cA(D)$.
Finally, $\cF_7\subseteq D\cup \cA(D)$ since for all $z\in\rest\cup A_0$ 
we have that $(z,v_\ell')\notin D$, by Property $(R4)$ and since $v_\ell'=u_S\notin B$. 

To bound $t$ by the bias $b$ we note that
\begin{align*}
t\leq \sum_{i=1}^7 |\cF_i| 
	& \leq |B| + (|A|+1) + r + (2\ell - 1) + (r-\ell +1) + e_D(v_{\ell +1}',\rest\cup A_0)\\
	& = |A| + |B| + 2r + \ell + 1 + e_D(v_{\ell},\rest) + e_D(v_{\ell},A_0)\\
	& \leq |A| + |B| + 2r + \ell + 1 + \max\{\size{A},\size{B}\} + (r+1-\ell)\\
	& \leq 3 \max\{\size{A},\size{B}\} + 3r + 2
\end{align*}
where the third inequality follows from Property $(R3.1)$. 

Now, by Proposition \ref{obs:trivial} \eqref{eq:sizeA} and since $r\leq n/25-1$, 
it follows that for $n>8$,\\
$$ t\leq \frac{3n}{5}+\frac{3n}{25}+2 \leq \frac{19n}{20} \leq b.$$

We now show that $D'=D\cup \{e,f_1,\ldots, f_t\}$ is a riskless digraph of rank $r+1$. 

Set 
$A_S':=\{v_1',\ldots,v_{r+1}'\}=A_S\cup \{u_S\}$, 
$A_0':=(A_0\cup \{u_A\})\setminus \{u_S\}$ and 
$A':=A\cup \{u_A\}$,
and 
$B_S':=\{w_1',\ldots,w_{r+1}'\}=B_S\cup \{y_B\}$, 
$B_0' :=B_0$ and 
$B'=B\cup\{y_B\}.$
We claim that $(A',B')$ is a UDB in $D'$ with partitions $A'=A_0'\cup A_S'$ and 
$B'=B_0'\cup B_S'$ such that the Properties $(R1)$--$(R4)$ are satisfied for $r+1$ and 
such that $e_{D'}(\restStrich,w_1')=0$. 
Checking the properties is straight-forward and we want to guide the reader back 
to Figure \ref{fig:riskless1} and Figure \ref{fig10} before plunging into the technical details 
that follow. 

Since $(A,B)$ is a UDB in $D$ and $D\cup \cF_1\cup \cF_2\subseteq D',$
$(A',B')$ forms a UDB in $D'$. 
For Property $(R1)$, note that 
$\size{A'}=\size{A}+1$, $\size{B'}=\size{B}+1$ 
and $|A_S'|=|B_S'|=r+1$. 

For Property $(R2.1)$, note first that $A_S=A_S'\setminus \{v_{\ell}'\}$ 
induces a transitive tournament in $D\subseteq D'$. 
Furthermore, for all $i<\ell$ we have that $(v_i',v_{\ell}')\in \cF_5\subseteq D'$,
and for all $i>\ell$ we have that $(v_{\ell}',v_i')\in \cF_6\subseteq D'$.
Hence, $(v_1',\ldots,v_{r+1}')$ induces a transitive tournament in $D'$.
Furthermore, $B_S=B_S'\setminus \{w_1'\}$ induces a transitive tournament in $D\subseteq D'$, 
and for all $2\leq i\leq r+1$, we have that $(w_1',w_i')\in \cF_3\subseteq D'$. 
Hence, $(w_1',\ldots,w_{r+1}')$ induces a transitive tournament in $D'$.

For $(R2.2)$, let $z\in A_0'\cup\restStrich$. 
Then note that only arcs from $D\cup \{e\} \cup \cF_4 \cup \cF_7$ 
contribute to the set $\{i: (v_i',z)\in D'\}$.  
Moreover, $D(v_{\ell}',A_0'\cup\restStrich)=\emptyset$
by Property $(R4)$ and since $v_{\ell}'\in A_0\cup \rest$.
Note also that $\{i: (v_i,z)\in D\}$ is a down set of $[r]$ and the relative order of the $v_i'$ for $i\neq \ell$ 
did not change. 
If $z\neq w$, then the arcs from $\cF_7$ reestablish the down-set property for $D'$.  
If $z=w$, then the arcs from $\cF_4\cup \{e\}$ reestablish the down-set property for $D'$. 

For $(R2.3)$, let $z\in B_0'\cup \restStrich$. 
Then note that only arcs from $D$ 
contribute to the set $\{i: (z,w'_i)\in D'\}$.  
Moreover, $\{i: (z,w_i)\in D\}$ is an upset of $[r]$ by assumption. 
Further, note that $w_{i+1}'=w_i$ for every $1\leq i\leq r$,
and that $(z,w_1')=(z,y_B)\not\in D$ by Property $(R4)$ and since $y_B\in\rest$.
Thus, $\{i: (z,w_i')\in D'\}$ is an upset of $[r+1].$

For $(R3.1)$, let first $1\leq i \leq \ell -1$. Then only the arcs in $D(v_i,A_0)\cup\cF_4$ 
contribute to $D'(v_i',A_0')$; and only the arcs in $D(v_i,\rest)\cup\cF_4$ 
contribute to $D'(v_i',\restStrich)$. 
Therefore, $e_{D'}(v_i',A_0')\leq e_{D}(v_i,A_0) + 1\leq (r+1)+1-i$ and 
$e_{D'}(v_i',\restStrich)\leq e_{D}(v_i,\rest)+1\leq \max\{|A|,|B|\}+1=\max\{|A'|,|B'|\}$. \\
Now, let $i=\ell$. Observe that
$e_D(v_{\ell}',A_0\cup\rest)=0$ by $(R4)$ and since $v_{\ell}'\in A_0\cup \rest$.
So, only the arcs in $\cF_7\cup \{e\}$ 
contribute to $D'(v_{\ell}',A_0')$ and $D'(v_{\ell}',\restStrich)$. 
Therefore, $e_{D'}(v_{\ell}',A_0')\leq e_D(v_{\ell},A_0)+1\leq (r+1)+1-{\ell}$
and $e_{D'}(v_{\ell}',\restStrich)\leq e_{D}(v_{\ell},\rest)+1\leq \max\{|A|,|B|\}+1=\max\{|A'|,|B'|\}$.\\
Finally, let $\ell+1\leq i\leq r+1$. 
Then $D'(v_i',A_0')=D(v_{i-1},A_0)$ and 
$D'(v_i',\restStrich)\subseteq D(v_{i-1},\rest)$. 
Hence, $e_{D'}(v_i',A_0')\leq (r+1)+1-i$ and $e_{D'}(v_i',\restStrich)\leq \max\{|A'|,|B'|\}$.

For $(R3.2)$, first let $2\leq i\leq r+1$. Then only the arcs in $D(B_0,w_{i-1})$ 
contribute to $D'(B_0',w_i')$; and only the arcs in $D(\rest,w_i')$ 
contribute to $D'(\restStrich,w_i')$. 
Therefore we have
$e_{D'}(B_0',w_i')\leq i-1$ and $e_{D'}(\restStrich,w_i')\leq \max\{\size{A},\size{B}\}\leq \max\{\size{A'},\size{B'}\}$. \\
Now, for $i=1$ we have $w_1'=y_B\in\rest$. 
Similarly, only the arcs in $D(B_0,y_B)$ 
contribute to ${D'}(B_0',w_1')$; and only the arcs in $D(\rest,y_B)$ 
contribute to ${D'}(\restStrich,w_1')$. Then by Property $(R4)$ for the digraph 
$D$ we know $e_{D'}(B_0',w_1')= e_{D}(B_0,y_B)=0$ and analogously
$e_{D'}(\restStrich,w_1')=0.$ 
 
For $(R4)$, note that it is enough to prove that
$D'\subseteq D'(A',B')\cup D'(A_S',V)\cup D'(V,B_S').$
This indeed holds, since
\begin{align*}
D(A,B)\cup \cF_1\cup \cF_2 & \subseteq D'(A',B'),\\
D(A_S,V\setminus B)\cup \cF_4\cup \cF_5\cup \cF_6\cup \cF_7 & \subseteq D'(A_S',V), \text { and}\\
D(V\setminus A,B_S)\cup \cF_3 & \subseteq D'(V,B_S').\qedhere
\end{align*}
\end{proof}

Next, we prove Lemma \ref{lem:AddEdgesI} which ensures that OBreaker
can add arcs to a riskless graph without destroying its structural properties.

\begin{proof}[Proof of Lemma \ref{lem:AddEdgesI}] 
We first provide OBreaker with a strategy to direct  
$(r+1)(b+1)- \size{D}$ available arcs, then we show that 
OBreaker can follow that strategy, and that the resulting 
digraph $D'$ is riskless of rank $r+1$.
 
Initially, set 
$t=(r+1)(b+1)- \size{D}$ and let $\cF:=\emptyset$ be the (dynamic) 
set of OBreaker's edges. We proceed iteratively: 
As long as $t\geq \max\{|A|,|B|\}$, OBreaker enlarges $A_0$ and $B_0$ (and thus $A$ and $B$) 
alternately. As soon as $t<\max\{|A|,|B|\}$, he 
fills up the stars with centres $w_i$, $1\leq i\leq r$, starting with $w_{r+1}$.

{\bf Step 1: $t\geq \max\{|A|,|B|\}$.}  
If 
$|B| = |A|-1$, then let $y_B\in \rest$ such that $e_{D\cup \cF}(y_B,B)=0$. 
For all $x\in A$, if the arc $(x,y_B)\not\in D\cup \cF$, 
OBreaker directs $(x,y_B)$, 
decreases $t$ by one and updates $\cF:=\cF \cup \{(x,y_B)\}$. 
Set $B:=B\cup \{y_B\}$, $B_0:=B_0\cup \{y_B\}$ and repeat Step 1. 

If $|B|\geq |A|$, then let $x_A\in \rest$ such that $e_{D\cup \cF}(A,x_A)=0$. 
For all $y\in B$, if the arc $(x_A,y)\not\in D\cup \cF$, 
OBreaker directs $(x_A,y)$, updates $\cF:=\cF \cup \{(x_A,y)\}$ 
and decreases $t$ by one. 
Set $A:=A\cup \{x_A\}$, $A_0:=A_0\cup \{x_A\}$ and repeat Step 1.

\medskip
{\bf Step 2: $t<\max\{|A|,|B|\}$.} 
If $t=0$, there is nothing to do. Otherwise, OBreaker proceeds as follows. 

If $e_{D\cup \cF}(\rest, w_{r+1})<\max\{|A|,|B|\}$, then let $z\in \rest$ such that $e_{D\cup \cF}(z,B)=0$. 
Then OBreaker directs $(z,w_{r+1})$, 
updates $\cF:=\cF \cup \{(z,w_{r+1})\}$, 
decreases $t$ by one and repeats Step 2. 

Otherwise, $e_{D\cup \cF}(\rest, w_{r+1})=\max\{|A|,|B|\}$. Then let $\ell$ be the maximal index $i\in [r]$ such that 
$e_{D\cup \cF}(\rest, w_i)<\max\{|A|,|B|\}$. Let $z\in \rest$ such that $(z,w_\ell)\not\in D\cup \cF$ and 
$(z,w_{\ell+1})\in D\cup \cF$. OBreaker directs $(z,w_\ell)$, updates $\cF:=\cF \cup \{(z,w_{\ell})\}$, 
decreases $t$ by one and 
repeats Step 2. 

\medskip

We first show that OBreaker can follow the strategy. 
First note, by Property $(R4)$ of a riskless digraph, that for all $z\in \rest$, 
for all $x\in A$ we have that $(z,x)\not\in D$, and for all $y\in B$ we have that 
$(y,z)\not\in D$. Hence, under the assumption that $x_A$ and $y_B$ in Step 1 exist, 
OBreaker can follow the proposed strategy in Step 1. 

Now, since $D$ is a riskless digraph of rank $r+1$ with $|D|\leq (r+1)(b+1)$, 
and since $r\leq n/25-1$, by Proposition \ref{obs:trivial} \eqref{eq:sizeA} 
we have that 
\begin{align*}
\size{X_A}&= \size{\{z\in \rest: e_D(A,z)= 0\}}>2n/5-2 \text{ and}\\
\size{Y_B}&= \size{\{z\in \rest: e_D(z,B)= 0\}}>2n/5-2, 
\end{align*} 
before the first update in Step 1. 
On the other hand, 
in each iteration of Step 1, 
$A$ or $B$ increases by one vertex from $\rest$ (alternatingly).
Since by $(R4)$ there are no arcs inside $\rest$,
and since in Step 1, OBreaker directs all edges between
the new vertices in $A$ and those in $B$,
the size of each of these two sets can increase by at most $\sqrt{b}\leq \sqrt{n}\leq n/100$ 
for large enough $n$. 
Since $X_A$ and $Y_B$ consist of at least $2n/5-2$ elements each before entering 
Step 1, the existence of $x_A$ and $y_B$ in each iteration of Step 1 follows. 

For Step 2, the existence of $z\in \rest$ such that $e_{D\cup \cF}(z,B)=0$ 
is guaranteed by the following: 
Consider a vertex $z$ in $Y_B$ before Step 1. 
Then $e_{D}(z,B)=0$ by definition, and $e_{D}(z,\rest)=0$ by $(R4)$. 
Now, if $z$ is not added to $A$ or $B$ during Step 1, then $e_{D}(z,B)=0$ holds still after the 
update of $B$. Since in Step 1, $\cF$ contains only arcs between $A$ and $B$, it follows, 
under the assumption that $z\in\rest$ after the update, that $e_{D\cup \cF}(z,B)=0$ before 
entering Step 2. 
Since $\size{Y_B}> 2n/5-2$ before entering Step 1, and since 
in Step 1 at most $2\sqrt{n}$ vertices are moved from $Y_B$ to $A\cup B$, 
if follows that at the beginning of Step 2, there are at least $n/4$ vertices $z$ 
in $\rest$ such that  $e_{D\cup \cF}(z,B)=0$. 
Note that in 
Step 2 
at most $\max\{|A|,|B|\}-1\leq n/5+\sqrt{n}<n/4$ 
of those vertices $z\in \rest$ with $e_{D\cup \cF}(z,B)=0$ are used. 
The existence of $\ell$ is always guaranteed since $e_D(\rest,w_1)=0$ by assumption, 
and since Step 2 is executed at most $\max\{|A|,|B|\}-1$ times. 
Note that by choice of $x\in \rest$, OBreaker can always direct $(x,w_{r+1})$ or $(x,w_{\ell})$ 
as required.

Finally, we prove that $D':=D\cup \cF$ is a riskless digraph of rank $r+1$, 
where $\cF$ is the set of all arcs that OBreaker directed in Step 1 and Step 2. 
In Step 1, the sets $A$ and $B$ are enlarged by one in each iteration. 
Since for each new element $x_A$ (or $y_B$, respectively) all arcs 
$(x_A,y)$ for $y\in B$ (or $(x,y_B)$ for $x\in A$, respectively) are directed by 
OBreaker (unless they are in $D\cup \cF$ already), the pair $(A,B)$ is a UDB in $D\cup \cF$. 
Since $A$ and $B$ are increased alternately 
(except for the first execution of Step 1 in case $|B|=|A|+1$), 
it follows that $\size{|A|-|B|}\leq 1$. 
Since $A_S$ and $B_S$ are unchanged, Property $(R1)$ follows.

Since $A_S$ and $B_S$ are untouched, there is nothing to prove for $(R2.1)$.
For $(R2.2)$, note that, after the last update of Step 2, for all $z\in A_0\cup \rest$, 
the set $\{i: (v_i,z)\in D\cup \cF\}$ is the same as $\{i: (v_i,z)\in D\}$.
Now, for all $z\in B_0\cup \rest$, the arc $(z,w_i)$ is directed by OBreaker 
for some $1\leq i<r+1$ only if $(z,w_{i+1})\in D\cup \cF$.
So $(R2.3)$ follows as well. 
For $(R3.1)$, note that in Step 1, all vertices $z$ that are added to $A_0$
fulfill $e_D(A,z)=0$. Hence for all $1\leq i \leq r+1$, $e(v_i,A_0)$ does not 
(strictly) increase
when proceeding from $D$ to $D\cup \cF$. Also, $e(v_i,\rest)$ does not (strictly) increase, since all arcs of the form $(v_i,z)$, 
that are directed by OBreaker, fulfill $z\in B$ after the update. In Step 2, only edges of the form 
$(z,w_i)$ for $z\in \rest$ are directed, hence $(R3.1)$ follows. 
For $(R3.2)$, similar to $(R3.1)$, the quantity $e(B_0,w_i)$ does not increase in Step 1, 
for all $1\leq i\leq r+1$, since all vertices $z\in \rest$ added to $B_0$ fulfill $e_D(z,B)=0$. 
In Step 2, no vertices are added to $B_0$, so the quantity $e(B_0,w_i)$ 
stays unchanged for all $1\leq i \leq r+1$. 
Now for $1\leq i\leq r+1$, the quantity $e(\rest,w_i)$ only increases in Step 2, 
and only if $e(\rest,w_i)<\max\{|A|,|B|\}$ 
by the strategy description. Therefore, $(R3.2)$ follows. 
Finally, Property $(R4)$ follows since OBreaker updates $A$ and $B$ accordingly in Step 1, 
and since in Step 2, he only directs arcs of the form $(x,w_i)$ for $x\in \rest$ and $w_i\in B_S$. 
\end{proof}

\section{Strict rules - Stage II} \label{sec:StrictStageII}

\begin{proof}[Proof of Lemma \ref{lem:Transition}]
By assumption, $D$ is a riskless digraph of rank $r$.  
Let $A=A_S\cup A_0$ and $B=B_S\cup B_0$ be given according to Definition \ref{def:riskless}.
We claim that $D$ is protected with UDB $(A,B)$ with partitions
$A= A_D\cup A_{AD}\cup A_{S}\cup A_{0}$ and $B= B_D\cup B_{AD}\cup B_{S}\cup B_{0}$, where  
$A_D=A_{AD}=B_D=B_{AD}=\emptyset$, 
and $k_1=\ell_1=0$, $k_2=\ell_2=r$.

For Property $(P1)$, 
let $a_0:=|A_0|$ and note that by Property $(R1)$, $\size{a_0-\size{B_0}}\leq 1$.
By assumption on $|D|$ and by Property $(R4)$, 
\begin{align}\label{aux110} 
r(b+1) = |D| = e_D(A,B) + e_D(A_S,V\setminus B)+ e_D(V\setminus A,B_S).
\end{align}
Now, $e_D(A,B) = (r+a_0)(r+|B_0|)\leq (r+a_0+1)^2$; whereas 
 \begin{align*}
 e_D(A_S,V\setminus B)&=e_D(A_S,A_S)+ e_D(A_S,A_0)+e_D(A_S,\rest)\\
	&\leq \binom{r}{2}+\frac{r(r+1)}{2} + r(a_0+1+r) \\
	&\leq r^2 + r(a_0+1+r)
 \end{align*} 
 where the first inequality follows from Properties $(R1)$ and $(R3.1)$. 
 Similarly, by Properties $(R1)$ and $(R3.2)$, 
 $e_D(V\setminus A, B_S)\leq r^2 + r(a_0+1+r).$ 
 Thus, \eqref{aux110} yields 
 $$r(b+1) \leq  (r+a_0+1)^2 + 4r^2 + 2r (a_0+1).$$ 
 Standard, but slightly tedious calculations show that this implies that 
 $a_0 + 1 \geq n/10+3$ for $r=\left\lfloor n/25 \right\rfloor$,  
 $b\geq 19n/20$ and $n$ large enough. 
 This then implies that $\size{B_0}\geq a_0-1\geq n/10+1$.

There is nothing to prove for Properties $(P2)$ and $(P3)$ 
since $A_{D}=A_{AD}=B_{D}=B_{AD}=\emptyset$.
For Property $(P4)$, note that $A_{AD}=B_{AD} =\emptyset$ and the enumerations  
$A_S=\{v_1,\ldots,v_r\}$ and $B_S=\{w_1,\ldots,w_r\}$ given by Property $(R2)$ 
fulfil $(P4.1)$--$(P4.3)$, with $k_1=\ell_1 = 0$ and $k_2=\ell_2=r$.
Properties $(P5.1)$ and $(P5.2)$ follow from $(R3.1)$ and $(R3.2)$, respectively. 
Finally, $(P6)$ follows from $(R4)$.
\end{proof}

\begin{proof}[Proof of Lemma \ref{lem:BaseII}]
Let $A=A_D\cup A_{AD}\cup A_S \cup A_0$ and  $B=B_D\cup B_{AD}\cup B_S \cup B_0$ 
be given by Definition \ref{def:protected}, and let 
$A_{AD}=\{v_1,\ldots,v_{k_1}\}$, $A_S=\{v_{k_1+1},\ldots,v_{k_1+k_2}\}$ and 
$B_{S}=\{w_1,\ldots,w_{\ell_2}\}$, $B_{AD}=\{w_{\ell_2+1},\ldots,w_{\ell_1+\ell_2}\}$ 
as given by Property $(P4)$.
Let $e=(v,w)\in \cA(D(V\setminus B))$ 
be given by the lemma.
For notational reasons we divide into two cases. 

{\bf Case 1: $v\in A_{AD}\cup A_S$.} Then $v= v_\ell$ for some $1\leq \ell \leq k_1+k_2$. 
Note that the only properties that may not be fulfilled anymore in $D+e$ 
are $(P4.2)$ and $(P5.1)$. 
Let $\{f_1,\ldots,f_t\} = (\cF_1\cup \cF_2 \cup \cF_3)\cap \cA(D)$, where 
\begin{align*}
\cF_1 &:= \big\{(v_i,w) : 1\leq i\leq \ell - 1\big\},\\  
\cF_2 &:= \big\{(v_1,z) : z\in A_0 \big\},\\ 
\cF_3 &:= \big\{(v_{k_1+1},z) : z\in\rest \big\}, 
\end{align*}
where we use the convention that $\cF_3=\emptyset$ if $A_S=\emptyset$. 
To show that this choice of arcs is suitable for the lemma, 
we now follow the structure of the proof of Lemma \ref{lem:BaseI}.
That is, we prove that $f_i\neq \overset{\leftarrow}{e}$ for all $1\leq i\leq t$, 
that $\cF_1\cup\cF_2\cup\cF_3 \subseteq D\cup \cA(D)$, that $t\leq b$, 
and finally we deduce that $D'=D\cup \{e,f_1,\ldots,f_t\}$ is protected.

For the first part, $\eback = (w,v) \not\in \cF_1$ since $v\neq w$, 
and $\eback = (w,v) \not\in \cF_2\cup \cF_3$ since $v\in A_{AD}\cup A_S$ by assumption. 
To see that $\cF_1\subseteq D\cup \cA(D)$, assume 
that $(w,v_i)\in D$ for some $1\leq i\leq \ell -1$.
Then $w=v_j\in A_{AD}\cup A_S$ for some $j< i <\ell$, by Properties $(P6)$,
$(P4.1)$ and since $w\notin A_D$. 
By Property $(P4.1)$ again, we conclude $(w,v)=(v_j,v_\ell)\in D$, 
a contradiction to $e\in \cA(D(V\setminus B)).$ 
Thus, $\cF_1 \subseteq D\cup \cA(D)$. 
Now, $\cF_2 \cup \cF_3 \subseteq D\cup \cA(D)$ 
since every arc of the form $(z,v_i)$ in $D$ satisfies
$z\in A\setminus A_0$, by Property $(P6)$. 
To see that $t\leq b$ note that 
$$ t\leq \size{\cF_1}+\size{\cF_2}+\size{\cF_3}
	\leq \ell +\size{A_0}+\size{\rest} 
	\leq \size{V\setminus B} 
	\leq b,$$ 
since by Property $(P1)$, $\size{B}\geq n/10+1\geq n-b$. 

To check that $D'$ is protected, consider the partition 
$A = A_D' \cup A_{AD}' \cup A_S' \cup A_0$ where 
$A_D':= A_D\cup \{v_1\}$,  
$A_{AD}':= \{v_2,\ldots , v_{k_1+1} \}$ (or $A_{AD}':= \{v_2,\ldots , v_{k_1} \}$ if $k_2=0$), and 
$A_S':=\{v_{k_1+2},\ldots,v_{k_1+k_2}\}$.
Clearly, $(A,B)$ is still a UDB in $D'$ with $|A|,|B|\geq n/10+1$, and Property $(P1)$ holds since 
$\size{A_D'\cup A_0}=\size{A_D \cup A_0}+1$. \\
For Property $(P2)$, $B_D$ did not change; and 
$D'(A_D')$ is a transitive tournament, since $D(A_D)\subseteq D'(A_D')$ 
is such, and since $(z,v_1)\in D\subseteq D'$ for every $z\in A_D$ by the
UDB-property for $A_D$. 
To see that $(A_D',V\setminus A_D')$ forms a UDB we need to observe that
$(v_1,z)\in D'$ for every $z\in V\setminus A_D'$. 
If $v_1\in A_{AD}$, then this follows by Properties $(P3)$ and $(P4.1)$ for $D$, 
and since $\cF_2\subseteq D'$. 
If $v_1\in A_{S}$ (and thus $k_1=\size{A_{AD}}=0$), then this follows since 
$(A,B)$ is a UDB in $D$, by Property $(P4.1)$ for $D$, 
and since $\cF_2\cup \cF_3\subseteq D'$. \\
To see that Property $(P3)$ holds in $D'$, observe first that 
$A_{AD}'\setminus \{v_{k_1+1}\}\subseteq A_{AD}$.  
Moreover, $(v_{k_1+1},z)\in D'$ for every $z\in V\setminus A$ since 
$(A,B)$ is a UDB in $D$ and since $\cF_3\subseteq D'$. 
For Property $(P4)$, it obviously holds that $\size{A_S'}<k_2\leq n/25$ 
and that $\size{B_S}\leq n/25$.  
Properties $(P4.1)$ and $(P4.3)$ follow trivially, Property $(P4.2)$ follows from 
Property $(P4.2)$ for $D$ and since $\cF_1\subseteq D'$. \\
For $(P5.1)$, observe that, since we made an index shift (from $A_S$ to $A_S'$), 
we have to prove that
$e_{D'}(v_{(k_1+1)+i},A_0)\leq n/25+1-i$ for every $1\leq i< k_2$.
First, let $1\leq i < k_2$ be such that 
$(k_1+1)+i\leq \ell$. Then only arcs from $D(v_{k_1+1+i},A_0)\cup \cF_1\cup\{e\}$ contribute 
to $D'(v_{k_1+1+i},A_0)$. 
Therefore, $e_{D'}(v_{k_1+1+i},A_0)\leq e_{D}(v_{k_1+1+i},A_0) + 1
\leq (n/25+1-(1+i))+1.$ 
Now let $k_1+1+i>\ell$. Then 
$D'(v_{k_1+1+i},A_0)=D(v_{k_1+1+i},A_0)$,  
and therefore, $e_{D'}(v_{k_1+1+i},A_0)\leq n/25+1-i$. \\
There is nothing to prove for Property $(P5.2)$. 
Finally, Property $(P6)$ follows since 
$\cF_1\cup \cF_2\cup \cF_3 \subseteq D'(A\setminus A_0, V\setminus B)$ and therefore, 
\begin{align*}
D' &= D\cup \cF_1\cup \cF_2\cup \cF_3 \\
	&= D(A,B) \cup D(A\setminus A_0, V\setminus B)\cup D(V\setminus A, B\setminus B_0)
	\cup \cF_1\cup \cF_2\cup \cF_3\\
	&=D'(A,B) \cup D'(A\setminus A_0, V\setminus B)\cup D'(V\setminus A, B\setminus B_0). 
\end{align*}

{\bf Case 2: $v\not\in A_{AD}\cup A_S$.}
By Property $(P2)$ and since $e=(v,w)\in \cA(D(V\setminus B))$, 
we may assume that $v\not\in A_D$, i.e. $v\in A_0\cup \rest$. 
We now want to incorporate $v$ into the tournament $A_{AD}\cup A_S$.   
Set 
\begin{align*}
\ell :=
\begin{cases}
\min \{i:\ (v_i,v)\notin D\} & \text {if minimum exists}\\
k_1+k_2+1 & \text{otherwise,}
\end{cases}
\end{align*}
and 
\begin{align*} 
v_i' :=
\begin{cases}
v_i 	& 1\leq i\leq \ell -1\\
v 	& i=\ell\\
v_{i-1}	& \ell +1\leq i\leq k_1+k_2+1.\\
\end{cases}
\end{align*}
Consider the following families of arcs. 
\begin{align*}
	\cF_1 &:= \big\{(v_i',w) : 1\leq i\leq \ell - 1\big\}\\
	\cF_2 &:= \big\{(v_{\ell}',v_i') : \ell + 1\leq i\leq k_1+k_2+1 \big\}\\
	\cF_3 &:= \big\{(v_{\ell}',z) : z\in\rest\cup A_0 \text{ and }(v_{\ell +1}',z)\in D  \big\}\\
	\cF_4 &:= \begin{cases}
	\big\{(v_1',z) : z\in A_0 \text{ and }(v_{\ell +1}',z)\notin D \big\} & \text{ if }v\in A_0\\
	\big\{(v_\ell',y) : y \in B\big\} & \text{ if }v\in \rest
	\end{cases}\\
	\cF_5 &:= \big\{(v_{k_1+1}',z) : z\in\rest,\ z\neq v \text{ and }(v_{\ell +1}',z)\notin D \big\},
\end{align*}
where we use the convention that if 
$\ell = k_1+k_2+1$ (and thus $v_{\ell+1}'$ does not exist) 
then we take $\cF_3 =\emptyset$,  
$\cF_5 := \big\{(v_{k_1+1}',z) : z\in\rest,\ z\neq v\big\}$,
and $\cF_4 := \big\{(v_1',z) : z \in A_0\big\}$ when $v\in A_0$. 
We illustrate the arcs in $\cF_1\cup\cdots\cup \cF_5$ in Figure \ref{fig20}.

 \begin{figure}[h]
 \centering
\begin{subfigure}{0.48\textwidth}
	\centering
	 \includegraphics[scale=0.7]
	 {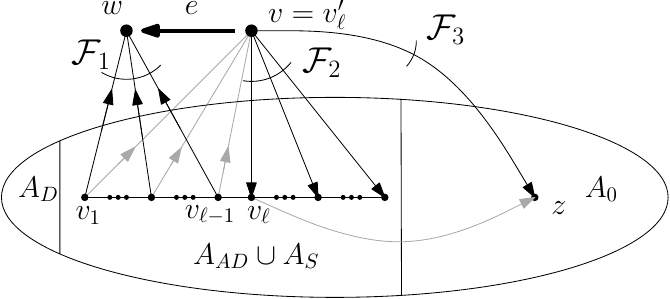}
\end{subfigure} 
\quad
 \begin{subfigure}{0.48\textwidth}
 	\centering
	\includegraphics[scale=0.5]
	{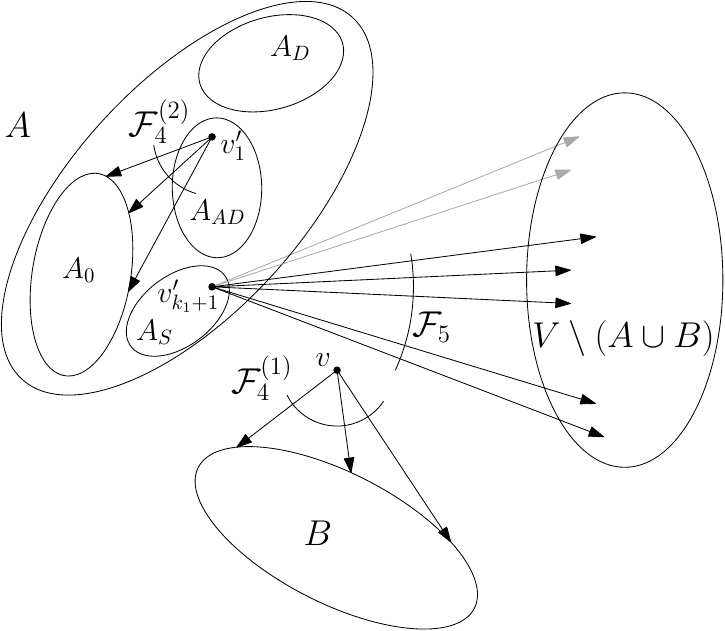}
 \end{subfigure}
 \caption{The case $v\in V\setminus (A\cup B)$. 
 Grey arcs are in $D$ already. 
 $\cF_1, \cF_2, \cF_3$ on the left 
 include $v$ into the tournament in $A$ 
 and restore $(P4.2)$. 
 On the right, 
 $\cF_5$ adds $v_{k_1+1}'$ to $A_{AD}$ while $(P3)$ is restored, and  
 $\cF_4^{(1)}$ adds $v$ to $A$. 
 (In case $v\in A_0$, $\cF_4^{(2)}$ adds $v_1'$ to $A_D$.)}
 \label{fig20}
 \end{figure}

We choose $\{f_1,\ldots,f_t\}$ to be $\{f_1,\ldots,f_t\} =\left( \cF_1\cup\cF_2\cup\cdots\cup \cF_5\right) \cap \cA(D).$
We proceed as before and show that $f_i\neq \overset{\leftarrow}{e}$ for all $1\leq i\leq t$, 
that $\cF_1\cup\cdots\cup\cF_5 \subseteq D\cup \cA(D)$, that $t\leq b$, 
and finally we deduce that $D'=D\cup \{e,f_1,\ldots,f_t\}$ is protected.

For the first part, $\eback = (w,v) \not\in \cF_1$ since $v\neq w$. 
Similarly, $\eback = (w,v) \not\in \cF_2 \cup \cF_3$ since $v_{\ell}'=v \neq w$. 
For the same reason, $\eback \not\in \cF_4$ in the case when $v\in \rest$. 
In the case when $v\in A_0$, assume that $(w,v)=(v_1',z)$ for some $z\in A_0$. 
Then $(v_1',v)=(w,v)\in \cA(D)$, by assumption on $e$, and $v_1'=v_1$. 
That is, $(v_1',v)\not\in D$, and $\ell =1$ by definition of $\ell$. 
But then $w=v_1'=v_{\ell}'=v$, by definition of $v_i'$, a contradiction. 
Also, $\eback \not\in \cF_5$ by definition of $\cF_5$.  
So, $f_i\neq \overset{\leftarrow}{e}$ for all $1\leq i\leq t$.

To see that $\cF_1\subseteq D\cup \cA(D)$, assume 
that $(w,v_i')=(w,v_i)\in D$ for some $1\leq i\leq \ell -1$.
Then $w=v_j\in A_{AD}\cup A_S$ for some $j< i <\ell$, by Properties $(P6)$,
$(P4.1)$ and since $w\notin A_D$.
By definition of $\ell$ we conclude $(w,v)=(v_j,v)\in D$,
a contradiction to $e\in \cA(D(V\setminus B)).$ 
Thus, $\cF_1 \subseteq D\cup \cA(D)$. 
Now, $\cF_2 \subseteq D\cup \cA(D)$  
by Property $(P4.2)$ and the definition of $\ell$. 
Also, $\cF_3 \cup \cF_4\cup \cF_5 \subseteq D\cup \cA(D)$ 
since every arc of the form $(w',v_i')$ in $D$ satisfies
$w'\in A\setminus A_0$, by Property $(P6)$ and since $v_i'\in A\cup \rest$.

To bound the number $t$ by the bias $b$ 
note that in case $v\in A_0$,
\begin{align*}
t & \leq |\cF_1\cup \cF_2| + |\cF_3\cup \cF_4 \cup \cF_5| \\ 
	& \leq k_1+k_2 + |A_0\cup \rest | \\
	& = |A_{AD}\cup A_S| + |A_0\cup \rest |
	\leq |V\setminus B|
	< \frac{9n}{10}<b,
\end{align*}
since by Property $(P1)$ we have $|B|\geq n/10+1.$
When $v\in\rest$, 
then we bound 
\begin{align*}
t & \leq |\cF_1\cup \cF_2| + |\cF_3 \cup \cF_5|  + |\cF_4|\\
	& \leq |A_{AD}\cup A_S| + (\size{\rest} + e_D(v_{\ell+1}', A_0)) + |B| \\
	& = \size{V\setminus(A_0\cup A_D)} +e_D(v_{\ell+1}', A_0).
\end{align*}
Since $v\in \rest$, we have that  $\ell\geq k_1+1$, 
by Property $(P3)$ and definition of $\ell$. 
Since $v_{\ell+1}'=v_\ell$, it follows that 
$$  t \leq \size{V\setminus(A_0\cup A_D)} +e_D(v_{\ell+1}', A_0)
	\leq  \frac{9n}{10}+\frac{n}{25} < b,$$
by Properties $(P1)$, $(P5.1)$ and choice of $b$.

Finally, we show that $D'=D\cup \{e,f_1,\ldots, f_t\}$ is a protected digraph. 
Set
\begin{align*}
A_D':=
\begin{cases}
A_D\cup \{v_1'\} & \text{ if } v\in A_0,\\
A_D	& \text{ if } v\in\rest
\end{cases}
\qquad
A_{AD}':=
\begin{cases}
\{v_2',\ldots , v_{k_1+1}' \} & \text{ if } v\in A_0\\
\{v_1',\ldots , v_{k_1+1}' \} & \text{ if } v\in \rest
\end{cases}
\end{align*}
\begin{align*}
A_S':=\{v_{k_1+2}',\ldots,v_{k_1+k_2+1}'\}
\qquad
A_0':= A_0\setminus \{v\}
\qquad
A':=A\cup \{v\}.
\end{align*}

Moreover, let $B'=B$ with the same partition as for $B$.
Then $(A',B')$ is a UDB in $D'$, since $(A,B)$ is a UDB in $D\subseteq D'$
and since in case $v\in\rest$ we have $\cF_4\subseteq D'$.
For Property $(P1)$, $|B_D\cup B_0|\geq n/10$ and $|A'|,|B'|\geq n/10+1$ obviously hold.
Now, observe that $|A_D'|=|A_D|+1$ and $|A_0'|= |A_0|-1$
in case $v\in A_0$, while $A_D=A_D'$ and $A_0=A_0'$ in case $v\in\rest$.
Thus $|A_D'\cup A_0'|= |A_D\cup A_0|\geq n/10$. 

For Property $(P2)$, there is nothing to prove when $v\in\rest$,
since then the sets of dead vertices do not change. So, let $v\in A_0$.
Again $B_D$ does not change.
$D'(A_D')$ is a transitive tournament, since $D(A_D)\subseteq D'(A_D')$
is such, and since $(z,v_1')\in D\subseteq D'$ for every $z\in A_D$ by the
UDB-property for $A_D$.
To see that $(A_D',V\setminus A_D')$ forms a UDB we need to observe that
$(v_1',z)\in D'$ for every $z\in V\setminus A_D'$.

Assume first that $v_1'=v_1\in A_{AD}\cup A_{S}$. 
Then $(v_1',z)\in D \subseteq D'$ for every $z\in B$ since $(A,B)$ is a UDB; 
and for every $z\in (A_{AD}\cup A_S)\setminus \{v_1'\}$ by Property $(P4.1)$. 
For every $z\in A_0$, 
if $(v_{\ell +1}',z)\notin D$ 
(or $\ell = k_1+k_2+1$ where $v_{\ell +1}'$ does not exist), 
then $(v_1',z)\in \cF_4 \subseteq D'$; 
and if $(v_{\ell +1}',z)\in D$ 
then $(v_1',z)\in D\subseteq D'$ by Property $(P4.2)$. 
For $z\in \rest$, 
if $v_1'\in A_{AD}$ then $(v_1',z)\in D \subseteq D'$ by Property $(P3)$ for $D$; 
if $v_1'\in A_{S}$ (and thus $k_1=|A_{AD}|=0$) and if $(v_{\ell +1}',z)\notin D$ 
(or $\ell = k_1+k_2+1$ where $v_{\ell +1}'$ does not exist),  
then $(v_1',z)\in \cF_5 \subseteq D'$; 
and if $v_1'\in A_{S}$ and $(v_{\ell +1}',z)\in D$,  
then $(v_1',z)\in D\subseteq D'$ by Property $(P4.2)$. 

Now, assume that $v_1'=v \in A_0$ and thus $\ell =1$. 
If $v_2'=v_{\ell + 1}'\in A_{AD}$, 
then $(v_2',z)\in D$ for every $z\in\rest$ by Property $(P3)$. 
Therefore, for every $z\in (A_0\cup \rest)\setminus\{v\}$ we have that 
$(v_1',z)\in \cF_3\cup \cF_4 \subseteq D'$. 
For every $z\in A_{AD}\cup A_S$, we have that $(v_1',z)\in \cF_2\subseteq D'$. 
Furthermore,  $(v_1',z)\in D\subseteq D'$ for every $z\in B$ since $v_1'\in A_0$ and $(A,B)$ 
is a UDB in $D$. 
If $v_2'\in A_{S}$ (or $v_2'$ does not exist) and therefore $k_1=0$, 
then similarly $(v_1',z)\in \cF_2\cup \cF_3\cup \cF_4 \cup \cF_5 \cup D \subseteq D'$ 
for all $z\in V\setminus A_D'$.  

For Property $(P3)$, observe that the statement for $B_{AD}$ does not change. 
To see that $(A_{AD}',V\setminus A')$ forms a UDB 
we need to observe that $(v_{k_1+1}',z)\in D'$ for every $z\in V\setminus A'$. 
If $z\in B'$, then this is clear,
since $(A',B')$ forms a UDB as we showed already above. 
Let now $z\in \restStrich=V\setminus(A\cup B\cup\{v\})$. 
If $\ell\leq k_1$, then $v_{k_1+1}'=v_{k_1}\in A_{AD}$. Therefore, $(v_{k_1+1}',z)\in D\subseteq D'$ 
by Property $(P3)$ for $D$. 
If $\ell\geq k_1+1$, 
then $(v_{k_1+1}',z)\in \cF_3\cup\cF_5\cup D \subseteq D'$ 
(where we use Property $(P4.2)$ which says that if $(v_{\ell+1}',z)=(v_{\ell},z)\in D$ then 
$(v_{k_1+1},z)\in D$).

For Property $(P4)$, note that $\size{A_S'}=\size{A_S}=k_2 \leq n/25$.
For $(P4.1)$ observe that the statement for $\{w_1,\ldots, w_{\ell_1+\ell_2}\}$ does not change. 
To see that $(v_2',\ldots,v_{k_1+k_2+1})$ or $(v_1',\ldots,v_{k_1+k_2+1})$
induces a transitive tournament, 
note first that the vertex set without $v_{\ell}'$ induces a transitive tournament in $D\subseteq D'$.
We have $(v_i',v_{\ell}')=(v_i,v)\in D\subseteq D'$ for every $i\leq \ell -1$, by 
definition of $\ell$. Moreover, $(v_{\ell}',v_i')\in D'$ for every $i\geq \ell +1$, 
since $\cF_2\subseteq D'$.

For $(P4.2)$, let $z\in A_0'\cup\restStrich$. 
We show that $\{i: (v_i',z)\in D\}$ is a down set of $[k_1+k_2+1]$. 
Note that this then implies $(P4.2)$ also when $v\in A_0$ and thus 
$A_{AD}'\cup A_S' =\{v_2',\ldots,v_{k_1+k_2+1}\}$. 
Since $z\in A_0'\cup\restStrich$,  
only arcs from $D\cup \{e\}\cup \cF_1 \cup \cF_3 \cup \cF_4\cup \cF_5$ 
contribute to the set $\{i: (v_i',z)\in D'\}$. 
Note that $\{i: (v_i,z)\in D\}$ is a down set of $[k_1+k_2]$ 
and the relative order of the $v_i'$ for $i\neq \ell$ does not change. 
So, if $z\neq w$, then the arcs from $\cF_3$ reestablish the down-set property for $D'$.  
If $z=w$, then the arcs from $\cF_1$ 
reestablish the down-set property for $D'$. 
Now, $\cF_4$ contributes at most the element $\{1\}$ to $\{i: (v_i',z)\in D'\}$ which is 
of no harm. 
The family $\cF_5$ may contribute the element $\{k_1+1\}$ to $\{i: (v_i',z)\in D'\}$ 
for some $z\in \rest$. However, by Property $(P3)$, $(v_i,z)\in D$ 
for all $1\leq i \leq k_1$, so $\cF_5$ certainly does not destroy the down-set property.

There is nothing to prove for Property $(P4.3)$, 
since $B$ and $\{w_1,\ldots, w_{\ell_1+\ell_2}\}$ are unchanged. 

For $(P5.1)$, observe that, since we make an index shift (from $A_S$ to $A_S'$), 
we have to prove that
$e_{D'}(v_{(k_1+1)+i}',A_0')\leq n/25+1-i$ for every $1\leq i\leq k_2$.
First, let $1\leq i \leq k_2$ be such that 
$(k_1+1)+i<\ell$. Then only arcs from $D(v_{k_1+1+i},A_0)\cup \cF_1$ contribute 
to $D'(v_{k_1+1+i}',A_0')$. 
Therefore, $e_{D'}(v_{k_1+1+i}',A_0')\leq e_{D}(v_{k_1+1+i},A_0) + 1
\leq (n/25+1-(1+i))+1.$\\
Now, let $(k_1+1)+i=\ell$. 
Then $e_D(v_{\ell}',A_0)=e_D(v,A_0)=0$ since $v\in A_0 \cup \rest$ and by Property $(P6)$ for $D$. 
So, only $\cF_3\cup \{e\}$ contributes to $D'(v_{\ell}',A_0')$. 
Therefore, $e_{D'}(v_{\ell}',A_0') \leq e_D(v_{\ell+1}',A_0) + 1 = e_D(v_{\ell},A_0)+1
= e_D(v_{k_1+1+i},A_0)+1 \leq n/25+1-i.$ 
Finally, let $k_1+1+i>\ell$. Then $v_{k_1+1+i}'=v_{k_1+i}$ and
only arcs from $D(v_{k_1+i},A_0)$ contribute 
to $D'(v_{k_1+1+i}',A_0')$. This again proves $e_{D'}(v_{k_1+1+i}',A_0')\leq n/25+1-i$.

There is nothing to prove for Property $(P5.2)$. 

For $(P6)$ note that it is enough to prove that
$$D'\subseteq D'(A',B')\cup D'(A'\setminus A_0',V)\cup D'(V,B'\setminus B_0').$$
This indeed holds, since
\begin{align*}
D(A,B) & \subseteq D'(A',B'),\\
D(A\setminus A_0,V\setminus B)\cup \cF_1\cup \cdots \cup \cF_5 & \subseteq D'(A'\setminus A_0',V),\\
D(V\setminus A,B\setminus B_0) & \subseteq D'(V,B'\setminus B_0').\qedhere
\end{align*}
\end{proof}

\begin{proof}[Proof of Lemma \ref{lem:AddEdgesII}]
Assume first that $A_{AD}\neq \emptyset$. 
If $(v_1,y)\in D$ for all $y\in A_0$, 
where $v_1$ is the source of the tournament on $A_{AD}$ as before, 
then we may set $A_D:=A_D\cup\{v_1\}$ and $A_{AD}:= A_{AD}\setminus \{v_1\}$ 
and reapply the lemma. 
Otherwise, there exists a vertex $y\in A_0$ 
such that the pair $\{v_1,y\}$ is not directed. 
Then the arc $f=(v_1,y)$ obviously satisfies that 
$D+f$ is protected. \\
So assume from now on that $A_{AD}=\emptyset$. 
If $A_S \neq \emptyset$ 
let $v_1$ be the source of the transitive tournament $D(A_S)$. 
We similarly may assume 
that there exists a vertex $y\in \rest$ such that 
the pair $\{v_1,y\}$ is not directed, for otherwise we reapply
the lemma with $A_{AD}:=\{v_1\}$ and $A_S := A_S\setminus \{v_1\}$. 
Then the arc $f=(v_1,y)$ obviously satisfies that 
$D+f$ is protected.\\
So assume from now on that $A_{AD}=A_S=\emptyset$. 
If $A_0 \neq \emptyset$ and $\rest \neq \emptyset$, 
then let $f=(x,y)$ for some $x\in A_0$ and some $y\in \rest$.
Now set $A_S:=\{x\}$ and $A_0 := A_0\setminus \{x\}$. 
It is easy to see that all the Properties $(P2)$--$(P6)$ 
hold for $D+f$ after the update. 
For $(P1)$, note that by assumption, $|A_D\cup A_0|+1=|A|\geq n/10+1$
after the update. 
Hence, $|A_D\cup A_0|\geq n/10$.
If $|A_0| \geq 2$ and $\rest = \emptyset$, 
then let $f=(x,y)$ for some 
distinct vertices $x,y\in A_0$, and set $A_{AD}:=\{x\}$ 
and $A_0:=A_0\setminus \{x\}$. The properties follow as in the 
previous case. 
If $A_0=\{x\}$ is a singleton and $\rest = \emptyset$, 
set $A_D:=A_D \cup A_0$ and $A_0=\emptyset$ and reapply the lemma. \\
So we may assume that $A_{AD}=A_S=A_0=\emptyset$, that is, 
$A=A_D$ and thus $D(A)$ is a transitive tournament and $(A,V\setminus A)$ 
is a UDB. \\
Now, by a similar analysis, we can either pick a suitable arc $f$ in $V\setminus A$
or deduce that $B=B_D$, that is, $D(B)$ forms a transitive tournament 
and $(V\setminus B,B)$ is a UDB. 
By assumption, $D$ is not a transitive tournament on $K_n$, 
hence there must be an undirected pair $\{x,y\}$. 
Since $A=A_D$ and $B=B_D$, both $x,y$ must lie in $\rest$. 
We claim that $f=(x,y)$ is suitable: Set $A_S:=\{x\}$
and update $A:=A\cup \{x\}$. Note that 
since $(V\setminus B,B)$ is a UDB, all the arcs $(x,z)$ for $z\in B$ 
are elements of $D$. 
It is easy to see that the Properties $(P1)$--$(P6)$ hold for $D+f$ after the update. 
\end{proof}

\section{Concluding remarks and open problems} 

In Theorem \ref{OrientedCycleMain}, we provide a winning strategy for OBreaker in the 
monotone Oriented-cycle game when $b\geq  5n/6+2$ and thus 
prove that $t(n,\cC)\leq 5n/6+1$. 
On the other hand \cite{bks2012}, $n/2-2$ is the best known lower bound on $t(n,\cC)$. 
In the proof of this result, OMaker first builds a long directed path (of length $n-1$) in at most $n-1$ rounds. 
When $b\leq n/2-2$, OBreaker cannot have directed all 
``backward'' edges of this path, so OMaker can direct one of those 
and close a cycle. 
From the perspective of OBreaker, it is indeed most harmful if OMaker builds a long path. 
As we have seen in the motivation of $\alpha$-structures, OMaker can create $\ell - 1$ 
immediate threats in the $\ell^{th}$ round  with such a strategy. 
On the other hand, OBreaker can ensure that he needs to direct 
at most $\ell-1$ edges to answer every immediate threat, even if OMaker
plays another strategy than building a long path (cf.~Lemma~\ref{alphaProc}). 
Despite these seemingly matching strategies, a significant gap between lower and upper bound 
remains. 
Let us describe briefly why $b\geq n/\sqrt{2}-o(n)$ is necessary for our strategy to be a winning 
strategy for OBreaker. 

Recall that our strategy for OBreaker included to build a UDB $(A,B)$ such 
that both parts have size at least $n-b$. 
Suppose OBreaker succeeds to build a UDB $(A,B)$ of size $n-b-1$ in some round $r$ (in Stage I). 
Observe that then $r\geq (n-b-1)^2/(b+1)$. 
Suppose that OMaker only directs edges in $V\setminus B$ in the first $r$ rounds, so that  $k$, 
the size  of the $\alpha$-structure in $V\setminus B$, increases in each round; whereas $\ell$, 
the size  of the $\alpha$-structure in $V\setminus A$, is zero. 
Assume further that OBreaker only increases one of the values
$|A|-k=|A|-r$ and $|B|-\ell=|B|$ (in order to decrease the number of edges he has to direct in the next round). 
Without loss of generality, let this be $|A|-k$. 
In order to follow procedure $\alpha$
and to increase $|A|-k$ (by adding two vertices to $A$), OBreaker needs to direct
at least $k+2|B|\geq (n-b-1)^2/(b+1) + 2(n-b-1)$ edges in round $r+1$. 
But this is only possible if $b\geq n/\sqrt{2}-o(n).$
We conjecture that the correct threshold is asymptotically at least
$n/\sqrt{2}.$

\begin{conjecture}
For $b\leq n/\sqrt{2}-o(n)$ 
OMaker has a strategy to close a directed cycle in the monotone b-biased orientation game.
\end{conjecture}
It is of course desirable to determine the threshold bias for the monotone Oriented-cycle game $t(n,\cC)$ exactly. 
Concerning the strict rules, 
we wonder whether $t^+(n,\cC)$ and $t^-(n,\cC)$ are (asymptotically) equal. 

Here are two natural variants of the Oriented-cycle game. 

\medskip
\noindent
{\bf Playing on random graphs.\\}
Suppose we replace the edge set of the complete graph $K_n$ by the edges of a random graph $G\sim \cG_{n,p}$, for 
some $p=p(n)$. That is, OMaker and OBreaker only direct edges of $G$.  
OMaker wins if the final digraph (with underlying edge set of $G$) contains a directed cycle; otherwise, OBreaker 
wins. During the Berlin-Pozna\'n Seminar 2013, held in Gu\l{}towy, Poland, Tomasz {\L}uczak asked how the threshold bias behaves in this variant of the game. 

\begin{problem}
Given $0<p=p(n)<1$. 
What is the largest bias $b=b(n,p)$ such that 
OMaker asymptotically almost surely  has a strategy 
to create an oriented cycle in the $b$-biased orientation game played on the edge set of $G\sim \cG_{n,p}$, 
under the strict rules and under the monotone rules, respectively?
\end{problem}

Note that for any graph $G$ with maximum degree $\Delta$, 
a modified version of the trivial strategy shows that 
OBreaker can win the Oriented-cycle game played on the edge set of $G$ when $b\geq \Delta -1$. 
Therefore, $b(n,p)\leq (1+o(1)) np$, provided $p$ is large enough. 
As we show in this paper, this is not tight for $p=1$. 
Indeed, we believe that for smaller values of $p$ the trivial upper bound of $np$ is not tight as well.  
We want to remark that in general our strategy does not directly yield an improvement on this upper bound of $b(n,p)$. 

\medskip
{\bf Preventing cycles of fixed length.\\}
In a different direction, let $\cC_k$ be the property of a tournament to contain an induced copy of a (directed) 
cycle of length $k$. Note that the properties $\cC_3$ and $\cC$ are equivalent, 
since if a tournament contains a cycle (of some length $k$), then it contains a cycle of length three.  
Therefore, preventing a cycle of length $k$ when playing on $K_n$ is at least as easy for OBreaker as preventing a 
cycle of length three. Or in other words, $t(n,\cC_k)\leq t(n,\cC)$ for all $3\leq k\leq n$. 
Note that the property $\cC_k$ does not necessarily imply the property $\cC_{k-1}$, so the 
sequence $t(n,\cC_3)$, $t(n,\cC_4), \ldots $ is not necessarily monotone. 
The threshold bias for the Hamilton cycle game is $t(n,\cC_n)=(1+o(1))n/\ln n$, 
as was recently proved in \cite{bks2012}. 
We wonder how $t(n,\cC_k)$ behaves as a function of $n$ and $k$. In particular, it would be interesting to 
know for which values of $k=k(n)$ besides $n$ the function $t(n,\cC_k)$ is sublinear.

\bigskip
{\bf Acknowledgement.}
We would like to thank Tibor Szab\'o for invaluable comments on the manuscript and for simplifying the definition of $\alpha$-structures. 
Furthermore, we would like to thank Jean Cardinal,  Katarzyna Mieczkowska, Yury Person, Shakhar Smorodinsky and Bettina Speckmann for fruitful discussions. 
Finally, we would like to thank the anonymous referees for helpful comments. 
 
\bibliographystyle{abbrv}
\bibliography{OCycleRefs}

\end{document}